\documentclass[12pt]{amsart}

\usepackage[
text={440pt,575pt},
headheight=9pt,
centering
]{geometry}

\usepackage{amsmath, amssymb, amsthm, enumerate, bm, enumitem}
\allowdisplaybreaks % Exactly what is says on the tin.
\usepackage{mathtools}

\usepackage{graphicx}
\usepackage{subcaption}
\usepackage[normalem]{ulem}

\usepackage[all]{xy}

\usepackage{multirow}

\usepackage{mathptmx}

\usepackage{xcolor}
\usepackage{hyperref}
\definecolor{darkblue}{RGB}{0,0,160}
\hypersetup{
	colorlinks,%
	citecolor=darkblue,%
	filecolor=black,%
	linkcolor=darkblue,%
	urlcolor=darkblue
}

\usepackage[capitalise]{cleveref}
\crefname{claim}{claim}{claims}

\usepackage{verbatim}

\makeatletter
\newcommand{\nolisttopbreak}{\vspace{\topsep}\nobreak\@afterheading}
\makeatother

\theoremstyle{definition}
\newtheorem{definition}{Definition}[section]

\newtheorem{theorem}[definition]{Theorem}
\newtheorem{proposition}[definition]{Proposition}
\newtheorem{lemma}[definition]{Lemma}
\newtheorem{corollary}[definition]{Corollary}

\newtheorem*{fact*}{Fact}
\newtheorem{remark}[definition]{Remark}
\newtheorem{example}[definition]{Example}

\newtheorem{problem}[definition]{Problem}

\newcommand{\C}{\mathcal{C}}

\newcommand{\E}{\mathcal{E}}

\newcommand{\Hc}{\mathcal{H}}

\newcommand{\M}{\mathcal{M}}

\newcommand{\NN}{\mathbb{N}}
\newcommand{\Nk}{\mathfrak{N}}

\newcommand{\OO}{\mathbb{O}}
\newcommand{\RR}{\mathbb{R}}

\newcommand{\Pk}{\mathfrak{P}}

\newcommand{\ZZ}{\mathbb{Z}}

\newcommand{\Zk}{\mathfrak{Z}}

\newcommand\nub{{\boldsymbol 0}}
\newcommand\ab{{\boldsymbol a}}
\newcommand\bb{{\boldsymbol b}}

\newcommand\eb{{\boldsymbol e}}
\newcommand\epb{{\bm \epsilon}}

\newcommand\ub{{\boldsymbol u}}
\newcommand\vb{{\boldsymbol v}}
\newcommand\wb{{\boldsymbol w}}
\newcommand\xb{{\boldsymbol x}}
\newcommand\yb{{\boldsymbol y}}
\newcommand\zb{{\boldsymbol z}}

\DeclareMathOperator{\mndcl}{mn}

\DeclareMathOperator{\Sym}{Sym}

\DeclareMathOperator{\ind}{ind}

\DeclareMathOperator{\irr}{irr}

\DeclareMathOperator{\supp}{supp}

\DeclareMathOperator{\cone}{cone}

\DeclareMathOperator{\gp}{gp}

\DeclareMathOperator{\lin}{lin}

\DeclareMathOperator{\re}{red}

\DeclareMathOperator{\Hom}{Hom}
\DeclareMathOperator{\unit}{U}

\newcommand\defas{\coloneqq}

\makeatletter
\@namedef{subjclassname@2020}{%
	\textup{2020} Mathematics Subject Classification}
\makeatother

\title{
	Duality of monoids up to symmetry
}
\author[D. V. Le]{Dinh Van Le}
\address{Department of Mathematics, FPT University, Hanoi, Vietnam}
\email{dinhlv2@fe.edu.vn}

\author[T. R\"{o}mer]{Tim R\"{o}mer}
\address{Institut f\"ur Mathematik, Universit\"at Osnabr\"uck, 49069 Osnabr\"uck, Germany}
\email{troemer@uos.de}

\author[N. T. Vien]{Nguyen Thi Vien}
\address{Institut f\"ur Mathematik, Universit\"at Osnabr\"uck, 49069 Osnabr\"uck, Germany}
\email{thi.vien.nguyen@uni-osnabrueck.de}

\begin{document}
	
	\begin{abstract}
		We study duality for monoids in $\ZZ^{(\NN)}$ that are invariant under the action of the infinite symmetric group $\Sym$. Our main result is an equivariant Minkowski--Weyl theorem for monoids. More precisely, we analyze the evolution of dual monoids along stabilizing $\Sym$-invariant chains and describe the eventual behavior of their equivariant Hilbert bases. 
		In addition, we develop a systematic study of structural properties of dual symmetric monoids, including a characterization of the duals of positive and non-positive monoids.
	\end{abstract}
	\maketitle

	\section{Introduction}
	
	The classical Minkowski--Weyl theorem is a cornerstone of convex geometry. It provides a dual characterization of polyhedral cones, identifying them both as conical hulls of finite sets and as intersections of finitely many linear halfspaces. A fundamental consequence is that the dual of a finitely generated cone is again finitely generated. While this theory is well understood in finite dimensions, its extension to objects \emph{up to symmetry}, that is, invariant under the infinite symmetric group $\Sym$, has only recently begun to take shape.
	
	A general framework for studying symmetric cones and monoids was introduced in \cite{KLR22}. Within this setting, equivariant analogues of the Minkowski--Weyl theorem were established for nonnegative cones in \cite{LR23}, and later in full generality in \cite{L25}. These developments are motivated by and connected to several active areas of research, including representation theory \cite{CEF15,CF13,NR19,SS17}, algebraic statistics \cite{AHT12,AH07,HM13,HS12}, and the theory of symmetric ideals in infinite-dimensional polynomial rings (see, e.g., \cite{D14,D22,JLR20,LNNR20,LNNR21,LR24,NR17}).
	
	Building on the systematic study of symmetric monoids in \cite{LRV26}, the present paper develops a duality theory for monoids up to symmetry. Our aim is to extend the duality theory for symmetric cones developed in \cite{L25,LR23} to the discrete setting of monoids, and in particular to establish equivariant analogues of the Minkowski--Weyl theorem in this context. This leads to the following guiding problem:
	
	\begin{problem}
		\label{pb:duality}
		Let $\ZZ^{(\NN)}$ denote the free abelian group with basis $\NN$, and let $\Theta$ be a subgroup of $\Sym$. Given a $\Theta$-invariant monoid $M \subseteq \ZZ^{(\NN)}$, describe its dual monoid $M^*$.
	\end{problem}
	
	We focus on two natural instances of this problem: the \emph{local} case, where $\Theta=\Sym(n)$ and $M\subseteq \ZZ^n$, and the \emph{global} case, where $\Theta=\Sym$ and $M\subseteq \ZZ^{(\NN)}$.

	We begin by exploring structural properties of $M^*$, such as normality and positivity. We show that $M^*$ is intimately related to the dual of the cone generated by $M$, and in particular is always normal (\Cref{prop:local-dualproperties,prop:global-dualproperties}). We also give an explicit description of $M^*$ when $M$ is non-positive (\Cref{duality: local non-positive,duality: global non-positive}). In the positive case, although the positivity of $M^*$ can be characterized (\Cref{duality: local positive,duality: global positive}), obtaining explicit generating sets becomes substantially more challenging.
	
	A key difficulty arises in the global setting, where the dual monoid $M^*$ is typically too large to handle directly (\Cref{M-W:global}). To circumvent this, we employ the local--global principles developed in \cite{KLR22, L25, LRV26}. Specifically, we associate to a $\Sym$-invariant monoid $M \subseteq \ZZ^{(\NN)}$ a chain of local monoids $\M=(M_n)_{n\ge 1}$, where $M_n = M \cap \ZZ^n$ is $\Sym(n)$-invariant. This reduces the study of $M^*$ to the asymptotic behavior of the sequence $(M_n^*)_{n\ge 1}$, leading to the central problem of the paper:
	
	\begin{problem}
		Let $\M=(M_n)_{n\ge1}$ be a $\Sym$-invariant chain of monoids. Describe the evolution of generating sets, and in particular Hilbert bases, of $M_n^*$ as $n$ increases.
	\end{problem}

	Our main results provide a solution to this problem. More precisely, we introduce an insertion operator that propagates a generating set of $M_t^*$ to generating sets of $M_n^*$ for all $n>t$, provided that $t$ is sufficiently large (\Cref{M-W:local general}). This yields an equivariant analogue of the Minkowski--Weyl theorem for symmetric monoids. Although analogous in spirit to the cone case \cite{L25, LR23}, the proof requires new ideas, reflecting the more intricate discrete structure of monoids.
	
	When the chain $\M$ consists of positive monoids, we show that the sequence of  dual monoids $\M^*=(M_n^*)_{n\ge 1}$ is eventually positive (\Cref{duality:chain of positive monoids}). Crucially, the equivariant Hilbert bases of $M_n^*$ admit an explicit description for all sufficiently large $n$ (\Cref{M-W:local 2r}). A key ingredient is that the insertion operator eventually preserves both irreducibility and reducibility of elements (\Cref{lem:irreducibility-presevation,lem:reducibility-presevation}).
	
	This work contributes to the broader program of developing a comprehensive theory of polyhedral geometry up to symmetry, extending classical finite-dimensional results to infinite-dimensional and equivariant settings.
	
	Let us now outline the structure of the paper. \Cref{sec:preliminaries} reviews the necessary background on symmetric groups and $\Sym$-invariant chains of monoids. In \Cref{sec:dual}, we introduce dual monoids and establish their basic properties. \Cref{sec:dualmonoids} studies duals of symmetric monoids, including positivity and the classification of non-positive cases. The equivariant Minkowski--Weyl theorem is proved in \Cref{sec:M-W}. Finally, \Cref{sec:equi-Hilbert-bases} provides explicit descriptions of equivariant Hilbert bases for dual sequences of chains of positive monoids.
	%-------------------------------------------------------
	\section*{Acknowledgements}
	The first author was supported by the Vietnam National Foundation for Science and Technology Development (NAFOSTED) under the grant number 101.04-2025.49. The second and third author were supported by the SPP 2458 “Combinatorial Synergies”, funded by the German Research Foundation under the grant number GZ: RO 2504/6-1; AOBJ: 704082. 
	
	%-------------------------------------------------------
	\section{Symmetric monoids}\label{sec:preliminaries}
	
	In this section we recall basic notions concerning symmetric groups and invariant chains of monoids. Our terminology follows \cite{KLR22,LR23}.
	
	Let $\NN$ denote the set of positive integers. Consider the \emph{Baer--Specker group} $\ZZ^\NN=\prod_\NN\ZZ$, which is the Cartesian power of $\ZZ$ indexed by $\NN$. As shown by Baer \cite{B37}, this abelian group is not free. It has the same cardinality as $\RR$ and is thus uncountable. Let
	\[
	\ZZ^{(\NN)}=\bigoplus_\NN\ZZ
	\]
	be the subgroup of $\ZZ^\NN$ consisting of elements with finite support, where the \emph{support} of an element $\ub=(u_i)_{i\in \NN}\in\ZZ^{\NN}$ is defined as
	\[
	\supp(\ub)=\{i\in\NN\mid u_i\ne 0\}.
	\]
	The group $\ZZ^{(\NN)}$ is free with basis $\{\eb_i\}_{i\in\NN}$, where $\eb_i$ denotes the $i$-th standard unit vector. Every element $\ub\in\ZZ^{(\NN)}$ therefore admits a unique representation $\ub=\sum_{i\in \NN}u_i\eb_i$ with $u_i\in\ZZ$ and $u_i=0$ for all but finitely many $i$. For such $\ub$, we define its \emph{coordinate sum}
	\[
	s(\ub)=\sum_{i\in\NN}u_i.
	\]
	
	For each $n\in\NN$, we identify $\ZZ^n$ with the subgroup of $\ZZ^{(\NN)}$ generated by $\{\eb_1,\dots,\eb_n\}$ via the inclusion  $(u_1,\dots,u_n)\mapsto(u_1,\dots,u_n,0,0,\dots)$. This identification yields an ascending chain of subgroups  
	\[
	\ZZ\subset \ZZ^2\subset\cdots\subset\ZZ^n\subset\cdots,
	\]
	such that $\ZZ^{(\NN)}=\bigcup_{n\ge1}\ZZ^n$.
	
	The real vector spaces $\RR^\NN$ and $\RR^{(\NN)}$ are defined analogously. In particular, the standard unit vectors $\{\eb_i\}_{i\in\NN}$ form a basis of $\RR^{(\NN)}$.
	
	\subsection{Symmetric group actions}
	
	For each $n\in\NN$, denote by $\Sym(n)$ the symmetric group on $[n]=\{1,\dots,n\}$. Viewing $\Sym(n)$ as the stabilizer of $n+1$ within $\Sym(n+1)$, we obtain an ascending chain
	\[
	\Sym(1)\subset \Sym(2)\subset \cdots \subset \Sym(n)\subset\cdots .
	\]
	Its union
	\[
	\Sym=\bigcup_{n\ge1}\Sym(n)
	\]
	consists precisely of those permutations of $\NN$ that move only finitely many elements.
	
	The group $\Sym$ acts naturally on $\RR^\NN$ via coordinate permutation:
	\[
	\sigma(\ub)=\big(u_{\sigma^{-1}(1)},u_{\sigma^{-1}(2)},\dots\big)
	\quad\text{for  }\ \sigma\in \Sym\ \text{ and }\ 
	\ub=(u_i)_{i\in \NN}\in\RR^\NN.
	\]
	In particular, $\sigma(\eb_i)=\eb_{\sigma(i)}$ for all $i\in\NN$. This action restricts to the subgroups $\RR^{(\NN)}$, $\ZZ^{\NN}$, and $\ZZ^{(\NN)}$, and induces the standard action of $\Sym(n)$ on $\RR^n$ and $\ZZ^n$ for each $n\ge1$.
	
	Let $\Theta$ be a subgroup of $\Sym$. A subset $A\subseteq\RR^{\NN}$ is called \emph{$\Theta$-invariant} if $\Theta(A)\subseteq A$, where
	\[
	\Theta(A)=\{\sigma(\ub) \mid \sigma \in \Theta,\; \ub \in A\}.
	\]
	denotes the $\Theta$-orbit of $A$. In the present work, we primarily focus on the cases $\Theta=\Sym$ and $\Theta=\Sym(n)$ for some $n\ge1$. 
	
	An ascending chain 
	$$A_1 \subset A_2 \subset \cdots \subset A_n \subset \cdots$$ 
	with $A_n \subseteq \RR^n$ is called $\Sym$-\emph{invariant} if each $A_n$ is $\Sym(n)$-invariant. The union $A = \bigcup_{n \ge 1} A_n$ of such a chain is evidently $\Sym$-invariant. Conversely, given a  $\Sym$-invariant set $A \subseteq \RR^{\NN}$, its truncations
	$$A \cap \RR \subset A \cap \RR^2 \subset \cdots \subset A \cap \RR^n \subset \cdots$$
	forms a $\Sym$-invariant chain. 
	
	As in \cite{KLR22,LR23}, we employ the term \emph{global} to refer to subsets of $\RR^{\NN},$ $\RR^{(\NN)},$ $\ZZ^{\NN},$ or $\ZZ^{(\NN)}$, whereas subsets of $\RR^n$ or $\ZZ^n$ are referred to as \emph{local}.
	
	\subsection{Monoids}
	Throughout the paper, a \emph{monoid} is a submonoid of $\ZZ^{\NN}$, that is, a subset closed under addition and containing the zero element $\nub$. Given a subset $A\subseteq \ZZ^{\NN}$, the monoid \emph{generated} by $A$ is defined as
	\[
	\mndcl(A)=\ZZ_{\geq 0} A\defas\left\{\sum_{i=1}^k m_i \ab_i \mid \ab_i \in A, m_i \in \ZZ_{\geq 0}, k \in \NN\right\}.
	\]
	A monoid admitting a finite generating set is called \emph{affine} or \emph{finitely generated}.

	Let $M\subseteq N$ be monoids. The \emph{saturation} of $M$ in $N$ is the following submonoid of $N$:
	$$\widehat{M}_N=\left\{u \in N \mid k \ub \in M \text { for some } k \in \NN\right\}.$$ 
	We say that $M$ is \emph{saturated} in $N$ if $M=\widehat{M}_N.$ Evidently, if $M$ is saturated in $N$, then it is also saturated in any monoid $N'$ with $M\subseteq N'\subseteq N$. The \emph{normalization} $\widehat{M}$ of $M$ is the saturation of $M$ in 
	\[
	\gp(M)=\ZZ M=\{\ub-\vb\mid \ub,\vb \in M\},
	\]
	the group generated by $M$, and $M$ is called \emph{normal} if $\widehat{M}=M$. The terminology stems from commutative algebra (see \cite[Section 4.E]{BG09}) and admits a geometric interpretation via cones. Recall that the \emph{cone} generated by a subset $A\subseteq \RR^{\NN}$
	is defined as
	\[
	\cone(A) =
	\RR_{\geq 0}A
	\defas
	\Big\{\sum_{i=1}^k\lambda_i\ab_i\mid\ab_i\in A,\
	\lambda_i\in\RR_{\geq 0}, \  k\in\NN\Big\}
	\subseteq\RR^{\NN}.
	\]
	
	\begin{lemma}
		\label{lem:Gordan}
		Let $M\subseteq N$ be monoids in $\ZZ^{\NN}$. Then
		\[
		\widehat{M}_N=\cone(M)\cap N.
		\]
		Moreover, if $N$ is an affine monoid in $\ZZ^n$ for some $n\ge 1$ and $\cone(M)$ is a finitely generated cone, then $\widehat{M}_N$ is also an affine monoid.
	\end{lemma}
	
	\begin{proof}
		The second statement is a slight generalization of the classical Gordan's lemma; see \cite[Lemma~2.9 and Proposition~2.22]{BG09}. The equality $\widehat{M}_N=\cone(M)\cap N$ is proved in \cite[Proposition~2.22]{BG09} when $N\subseteq\ZZ^n$. The same argument extends verbatim to $N\subseteq\ZZ^{\NN}$, and we omit the details.
	\end{proof}
	
	An element of a monoid $M$ is called a \emph{unit} if it has an additive inverse in $M$. The set of units of $M$ forms a group, denoted by $\unit(M)$. A non-unit element $\ub\in M$ is \emph{irreducible} if every decomposition $\ub = \vb + \wb$ with $\vb, \wb \in M$ forces either $\vb \in \unit(M)$ or $\wb \in \unit(M)$. We write $\irr(M)$ for the set of irreducible elements of $M$, and define
	$$\re(M)=M\setminus(\unit(M)\cup\irr(M))$$
	to be the set of \emph{reducible elements}.
	
	If $\unit(M)=\{\nub\}$, then $M$ is called \emph{positive}. This is the case, for instance, when $M \subseteq \ZZ_{\ge 0}^{\NN}$. When $M$ is positive, it is clear that $\irr(M)$ is contained in every generating set of $M$. If, in addition, $M$ is an affine monoid in $\ZZ^n$ for some $n\in \NN$, then it is generated by its irreducible elements (see \cite[Proposition 2.14]{BG09}), and hence $\irr(M)$ is the unique minimal generating set of $M$. In this case, $\irr(M)$ is referred to as the \emph{Hilbert basis} of $M$, denoted by $\Hc_M$.  
	
	We now turn to symmetric monoids. Let $\Theta$ be a subgroup of $\Sym$ and let $M$ be a $\Theta$-invariant monoid. To account for the symmetry of $M$, we consider generating sets up to the action of $\Theta$.
	
	\begin{definition}
		A subset $A\subseteq M$ is called a \emph{$\Theta$-equivariant generating set} of $M$ if the set of orbits $\Theta(A)$ generates $M$. We say that $M$ is \emph{$\Theta$-equivariantly finitely generated} if it has a finite $\Theta$-equivariant generating set.
	\end{definition}
	
	For example, while the monoid $\ZZ_{\ge 0}^{(\NN)}$ is not finitely generated in the classical sense, it is $\Sym$-equivariantly generated by the single unit vector $\eb_1$. 
	
	By exploiting symmetry, the notion of Hilbert bases can be extended to positive symmetric monoids in $\ZZ^{(\NN)}$, even those that are not affine. This relies on the next result. For the sake of brevity, when we state that a monoid $M\subseteq\ZZ^{(\NN)}$ is $\Theta$-invariant with $\Theta=\Sym$ or $\Theta=\Sym(n)$ for some $n\in \NN$, we refer to the global case where $M\subseteq\ZZ^{(\NN)}$ and $\Theta=\Sym$, or the local case where $M\subseteq\ZZ^{n}$ and $\Theta=\Sym(n)$.
	
	\begin{proposition}
		\label{prop:irreducible}
		Let $\Theta$ be a subgroup of $\Sym$ and $M$ a $\Theta$-invariant monoid. Then the set of irreducible elements $\irr(M)$ is $\Theta$-invariant. Moreover, if $M$ is a positive monoid in $\ZZ^{(\NN)}$ and $\Theta$ is either $\Sym$ or $\Sym(n)$ for some $n \in \NN$, then $\irr(M)$ is the unique minimal generating set of $M$.
	\end{proposition}
	
	To prove this result, we need the following characterization of positivity for symmetric monoids, established in \cite[Theorems 3.1 and 3.9]{LRV26}. Recall that $s(\ub)$ denotes the coordinate sum of $\ub$. 
	
	\begin{theorem}
		\label{thm:local positive and non-positive}
		Let $M\subseteq \ZZ^{(\NN)}$ be a $\Theta$-invariant monoid, where $\Theta=\Sym$ or $\Theta=\Sym(n)$ for some $n \in \NN$. Then $M$ is positive if and only if either $s(\ub)>0$ for all $\ub\in M\setminus\{\nub\}$, or $s(\ub)<0$ for all $\ub\in M\setminus\{\nub\}$.
	\end{theorem}
	
	\begin{proof}[Proof of \Cref{prop:irreducible}]
		The $\Theta$-invariance of $\irr(M)$ follows immediately from the fact that $\Theta$ acts by permuting coordinates. For the remaining assertion, it suffices to show that $\irr(M)$ generates $M$, because it is contained in every generating set of $M$. Let $\ub\in M\setminus\{\nub\}$. If $\ub\not\in\irr(M)$, then $\ub=\vb+\wb$ for some $\vb,\wb\in M\setminus\{\nub\}$. This yields $s(\ub)=s(\vb)+s(\wb)$. By \Cref{thm:local positive and non-positive}, we obtain $|s(\ub)|=|s(\vb)|+|s(\wb)|$, which implies $|s(\vb)|,|s(\wb)|<|s(\ub)|$. An induction on $|s(\ub)|$ shows that every $\ub\in M\setminus\{\nub\}$ is a finite sum of elements of $\irr(M)$. Therefore, $\irr(M)$ generates $M$, as desired.
	\end{proof}
	
	\Cref{prop:irreducible} allows us to make the following definition, extending the classical notion of Hilbert bases of affine monoids.
	
	\begin{definition}
		\label{def:equivariant-Hilbert-basis}
		Let $M\subseteq \ZZ^{(\NN)}$ be a positive $\Theta$-invariant monoid, where $\Theta=\Sym$ or $\Theta=\Sym(n)$. We define the \emph{Hilbert basis} of $M$ as $\Hc_M=\irr(M)$. A subset $H\subseteq M$ is called a \emph{$\Theta$-equivariant Hilbert basis} of $M$ if $H$ is a minimal set with respect to inclusion such that $\Hc_M=\Theta(H).$
	\end{definition}
	
	For instance, $\{\eb_i\mid i\in\NN\}$  is the Hilbert basis of $\ZZ_{\ge 0}^{(\NN)}$, and any set $\{\eb_i\}$ serves as a $\Sym$-equivariant Hilbert basis of this monoid.
	
	The study of global monoids $M \subseteq \ZZ^{(\NN)}$ often necessitates a transition to a chain of local monoids $\M=\left(M_n\right)_{n \geq 1}$ such that $M = \bigcup_{n \ge 1} M_n$. Various properties of $M$ can then be characterized 
	in terms of the chain $\M$. Such \emph{local-global principles} are central to the framework developed in \cite{KLR22,L25,LR23,LR24,LRV26}. Recall that an ascending chain of monoids $\M=\left(M_n\right)_{n \geq 1}$ is \emph{$\Sym$-invariant} if each $M_n$ is a $\Sym(n)$-invariant monoid in $\ZZ^{n}$. For any such chain, we always have
	$$\mndcl\left(\Sym(n)\left(M_m\right)\right) \subseteq M_n \quad \text{for all } n \geq m \geq 1.$$
	
	\begin{definition}
		A $\Sym$-invariant chain of monoids $\M=\left(M_n\right)_{n \geq 1}$ is said to \emph{stabilize} if there exists $r \in \NN$ such that
		$$M_n=\mndcl\left(\Sym(n)\left(M_m\right)\right) \quad \text {for all } n \geq m \geq r.$$
		The smallest such number $r$ is called the \emph{stability index} of $\M$, denoted by $\ind(\M)$.
	\end{definition}
	
	$\Sym$-invariant chains of cones and their stabilization are defined in an analogous manner. We refer the reader to \cite{KLR22, LR23} for a comprehensive treatment.
	
	%---------------------------------------------------------
	\section{Dual monoids}
	\label{sec:dual}
	
	This section is devoted to basic properties of dual monoids. In particular, we show that they are always normal. We also prove that the affine property is preserved under duality in the local setting.
	
	We begin by recalling the notion of dual cones. Let $C \subseteq \RR^{(\NN)}$ be a cone. The \emph{dual cone} $C^*$ is defined as the set of all $\RR$-linear forms $\alpha \in \Hom_{\RR}(\RR^{(\NN)},\RR)$ satisfying $\alpha(\ub) \geq 0$ for all $\ub \in C.$
	Identifying $\Hom_{\RR}(\RR^{(\NN)},\RR)$ with $\RR^\NN$ via the canonical pairing
	\[
	\langle \cdot , \cdot \rangle : \RR^{(\NN)} \times \RR^\NN \longrightarrow \RR, 
	\quad 
	\langle \ub, \vb \rangle = \sum_{i\ge 1} u_i v_i,
	\]
	the dual cone can be written as
	\[
	C^* = \left\{ \ub \in \RR^{\NN} \mid \langle \ub, \vb \rangle \geq 0 \ \text{ for all } \vb \in C \right\}.
	\]
	For a fixed $n\in \NN$ and a local cone $C_n \subseteq \RR^n$, we define 
	\[
	C_n^* = \left\{ \ub \in \RR^n \mid \langle \ub, \vb \rangle \geq 0 \ \text{ for all } \vb \in C_n \right\}.
	\]
	
	In the global setting, the finite generation of a cone does not necessarily pass to its dual. For instance, the dual of the cone $C=\{\nub\}\subseteq \RR^{(\NN)}$ is $C^*=\RR^{\NN}$, which is not finitely generated. However, for local cones, the situation is much better due to the classical Minkowski--Weyl theorem (see, e.g., \cite[Theorem 1.15]{BG09}):
	
	\begin{theorem}
		\label{thm:clasical-Minkowski-Weyl}
		A cone $C_n\subseteq\RR^{n}$ is finitely generated if and only if it is the intersection of finitely many closed linear halfspaces of $\RR^{n}$.
	\end{theorem}
	
	This result admits the following alternative formulation, which highlights the relationship between finite generation of a local cone and its dual. Since we are unaware of an appropriate reference, we sketch a proof for the convenience of the reader.
	
	\begin{theorem}
		\label{thm:dual-cone-fg}
		Let $C_n\subseteq\RR^{n}$ be a cone. Then the following are equivalent:
		\begin{enumerate}
			\item 
			$C_n$ is finitely generated;
			\item 
			$C_n^*$ is finitely generated and $C_n$ is closed in the standard topology on $\RR^{n}$.
		\end{enumerate}
	\end{theorem}
	
	\begin{proof}
		\textup{(i)}$\Rightarrow$\textup{(ii)}: 
		If $C_n$ is finitely generated, then by \Cref{thm:clasical-Minkowski-Weyl}, there exist $\ub_1,\dots,\ub_s\in\RR^n$ such that
		\[
		C_n=\left\{ \vb \in \RR^n \mid \langle \ub_i, \vb \rangle \geq 0 \ \text{ for all } i \in [s] \right\}.
		\]
		This implies that $C_n$ is closed. Moreover, $C_n^*=\cone(\ub_1,\dots,\ub_s)$ (see \cite[Theorem 1.16c]{BG09}), which is finitely generated. 
		
		\textup{(ii)}$\Rightarrow$\textup{(i)}: 
		Let $C_n^{**}$ denote the dual cone of $C_n^*$. Then $C_n=C_n^{**}$ since $C_n$ is closed; see, e.g., \cite[Theorem 1.16a]{BG09} and \cite[Theorem 11.5]{Ro70}. As $C_n^*$ is finitely generated, the implication \textup{(i)}$\Rightarrow$\textup{(ii)} ensures that $C_n^{**}$ is also finitely generated. 
	\end{proof}
	
	Let us now turn to dual monoids, which are defined analogously to dual cones. The pairing $\langle \cdot , \cdot \rangle$ restricts to $\ZZ^{(\NN)} \times \ZZ^\NN$. For a global monoid $M\subseteq \ZZ^{(\NN)}$, its \emph{dual monoid} is 
	$$M^* = \left\{ \ub \in \ZZ^{\NN} \mid \langle \ub, \vb \rangle \geq 0 \ \text{ for all } \vb \in M \right\},$$
	and for a local monoid $M_n \subseteq \ZZ^n$, we define
	$$M_n^* = \left\{ \ub \in \ZZ^n \mid \langle \ub, \vb \rangle \geq 0 \ \text{ for all } \vb \in M_n \right\}.$$

	The following result collects several basic properties of dual monoids. While these facts are largely standard, we include a proof for completeness.
	
	\begin{proposition}\label{prop:local-dualproperties}
		Let $M_n \subseteq \ZZ^{n}$ be a monoid, and let $\widetilde{M}_n$ denote its saturation in $\ZZ^n$. Then:
		\begin{enumerate}
			\item 
			$M_n^*=\cone(M_n^*)\cap \ZZ^n=\cone(M_n)^* \cap \ZZ^n.$
			In particular, $M_n^*$ is saturated in $\ZZ^n$, and hence normal.
			\item 
			$M_n^*=\widetilde{M}_n^*.$
			\item 
			$\widetilde{M}_n\subseteq M_n^{**}$, where $M_n^{**}$ denotes the bidual $(M_n^*)^*$. Equality holds if $M_n$ is affine.
		\end{enumerate}	
	\end{proposition}
	
	\begin{proof}
		(i) We first claim that
		\[
		\cone(M_n^*)\subseteq \cone(M_n)^*.
		\]
		Indeed, take $\ub =\sum_{i=1}^t \mu_i\ub_i\in \cone(M_n^*)$ and $\vb =\sum_{j=1}^s \eta_j\vb_j\in \cone(M_n)$, where $\ub_i\in M_n^*$, $\vb_j \in M_n$, and $\mu_i, \eta_j\in \RR_{\ge 0}$. 
		Since $\langle\ub_i, \vb_j \rangle \ge 0$ for all $i \in [t]$ and $j\in [s]$, we obtain
		$$\langle\ub, \vb \rangle = 
		\Big\langle \sum_{i=1}^t \mu_i\ub_i, \sum_{j=1}^s \eta_j\vb_j \Big\rangle 
		=  \sum_{i=1}^t\sum_{j=1}^s\mu_i \eta_j \langle\ub_i, \vb_j \rangle \ge 0.$$ 
		Hence, $\cone(M_n^*)\subseteq \cone(M_n)^*$, as claimed.
		It follows that
		\[
		M_n^*\subseteq \cone(M_n^*)\cap \ZZ^n \subseteq \cone(M_n)^* \cap \ZZ^n.
		\]
		For the reverse inclusion, let $\ub \in \cone(M_n)^* \cap \ZZ^n$. Then $\langle\ub, \vb \rangle \ge 0$ for all $\vb \in M_n$ because $M_n\subseteq\cone(M_n)$. As $\ub \in \ZZ^n$, this implies $\ub \in M_n^*$, establishing the first assertion.
		
		The equality $M_n^*=\cone(M_n^*)\cap \ZZ^n$ now implies, by \Cref{lem:Gordan}, that $M_n^*$ is saturated in $\ZZ^n$. In particular, it is normal since $\gp(M_n^*)\subseteq\ZZ^n$.
		
		(ii) Since $\cone(M_n)=\cone(\widetilde{M}_n)$, the assertion follows directly from (i).
		
		(iii) By (ii), $M_n^{**}=\widetilde{M}_n^{**}.$ Moreover, it is clear that $\widetilde{M}_n\subseteq \widetilde{M}_n^{**}$. Hence, $\widetilde{M}_n\subseteq M_n^{**}$. Now assume that $M_n$ is affine. Then $\cone(M_n)$ is finitely generated, and hence $\cone(M_n)^*$ is also finitely generated by \Cref{thm:dual-cone-fg}. Thus, there are elements $\ub_1,\dots,\ub_s$, which can be chosen in $\ZZ^n$ (see \cite[Proposition 1.69]{BG09}), such that $\cone(M_n)^*=\cone(\ub_1,\dots,\ub_s)$. By (i), we have
		\[
		\ub_i\in \cone(M_n)^*\cap \ZZ^n=M_n^*
		\quad\text{for all } i\in[s].
		\]
		Consequently, $\cone(M_n)^*\subseteq\cone(M_n^*)$. Together with the reverse inclusion from (i), this gives $\cone(M_n)^*=\cone(M_n^*)$. Applying (i) to $M_n^*$, we get
		\[
		M_n^{**}=\cone(M_n^*)^*\cap \ZZ^n
		=\cone(M_n)^{**}\cap \ZZ^n
		=\cone(M_n)\cap \ZZ^n
		=\widetilde{M}_n,
		\]
		where the equality $\cone(M_n)^{**}=\cone(M_n)$ follows from the fact that $\cone(M_n)$ is finitely generated (see \cite[Theorem 1.16]{BG09}).
	\end{proof}
	
	Without the affine assumption, the inclusions $\cone(M_n^*)\subseteq \cone(M_n)^*$ and $\widetilde{M}_n\subseteq M_n^{**}$ can be strict, as illustrated by the next example.
	
	\begin{example}
		\label{ex:strict-inclusion}
		Consider the monoid
		\[
		M_2=\{(u,v)\in\ZZ^2\mid u+v\sqrt{2}\ge0\}.
		\]
		We have
		\[
		\cone(M_2)=\{(x,y)\in\RR^2\mid x+y\sqrt{2}>0\}\cup\{(0,0)\}
		\]
		This cone is not closed, and its dual is
		\[
		\cone(M_2)^*=\{z(1,\sqrt{2})\mid z\ge 0\}.
		\]
		By \Cref{prop:local-dualproperties},
		\[
		M_2^*=\cone(M_2)^*\cap\ZZ^2=\{(0,0)\}.
		\]
		Thus,
		\[
		\cone(M_2^*)=\{(0,0)\}\subsetneq \cone(M_2)^*.
		\]
		Moreover, since $\widetilde{M}_2=\cone(M_2)\cap\ZZ^2=M_2$, one obtains
		\[
		\widetilde{M}_2\subsetneq\ZZ^2=M_2^{**}.
		\]
		Note that $\cone(M_2)$ is not finitely generated by \Cref{thm:dual-cone-fg}. Hence, $M_2$ is not affine. 
	\end{example}

	\Cref{prop:local-dualproperties} has the following global analogue.
	
	\begin{proposition}\label{prop:global-dualproperties}
		Let $M \subseteq \ZZ^{(\NN)}$ be a monoid, and let $\widetilde{M}$ denote its saturation in $\ZZ^{(\NN)}$. Then:
		\begin{enumerate}
			\item 
			$
			M^*=\cone(M^*)\cap \ZZ^{\NN}=\cone(M)^* \cap \ZZ^{\NN}.
			$
			In particular, $M^*$ is saturated in $\ZZ^{\NN}$, and hence normal.
			\item 
			$M^*=\widetilde{M}^*.$
			\item 
			$\widetilde{M}\subseteq M^{**}$, with equality if $M$ is affine.
		\end{enumerate}	
	\end{proposition}
	
	\begin{proof}
		All statements, except the last one that $M^{**}=\widetilde{M}$ when $M$ is affine, follow as in the local case. For the last statement, a different argument is required because the dual of a finitely generated cone in $\RR^{(\NN)}$ need not be finitely generated. Our approach relies on Specker's theorem \cite[Satz III]{Sp50} (cf. \cite[Chapter 13, Corollary 2.11]{Fu15}), which asserts that $\Hom(\ZZ^\NN,\ZZ)\cong\ZZ^{(\NN)}.$ This identification allows us to view the bidual $M^{**}$ as a subset of $\ZZ^{(\NN)}$. If $M$ is generated by $\vb_1,\dots,\vb_s\in\ZZ^{(\NN)}$, then all these generators lie in some $\ZZ^n$. Let $M_n\subseteq\ZZ^n$ be the monoid generated by $\vb_1,\dots,\vb_s$. By definition,
		\begin{align*}
			M^*&=\big\{\ub\in \ZZ^\NN\mid \langle\ub,\vb_k\rangle\ge0
			\text{ for all } k\in[s]\big\}\\
			&=\big\{\ub=(u_i)_{i\in\NN}\in \ZZ^\NN\mid (u_1,\dots,u_n)\in M_n^*\big\}.
		\end{align*}
		Now let $\wb=(w_i)_{i\in\NN}\in M^{**}$. Then  $\langle\wb,\ub\rangle\ge0$ for all $\ub\in M^*$. Since $u_j$ can take arbitrary integer values for $j\ge n+1$, this forces $w_j=0$ for all $j\ge n+1$. Hence $M^{**}\subseteq\ZZ^n$, and therefore $M^{**}=M_n^{**}$. By \Cref{prop:local-dualproperties}, we have
		\[
		M^{**}=\widetilde{M}_n=\cone(M_n)\cap \ZZ^n
		\subseteq \cone(M)\cap \ZZ^\NN=\widetilde{M}.
		\]
		Combining this with the general inclusion $\widetilde{M}\subseteq M^{**}$, we conclude that $M^{**}=\widetilde{M}$.
	\end{proof}
	
	Finally, we consider the behavior of the affine property under duality. As for cones, this property is generally not preserved in the global setting. For example, $\{\nub\}^*=\ZZ^{\NN}$ is not finitely generated. In the local case, however, we obtain the following partial analogue of \Cref{thm:dual-cone-fg}.
	
	\begin{proposition}\label{prop:local-affine}
		If $M_n \subseteq \ZZ^{n}$ is an affine monoid, then so is $M^*_n$. 
	\end{proposition}
	
	\begin{proof}
		If $M_n$ is affine, then $\cone(M_n)$ is a finitely generated cone. Hence, $\cone(M_n)^*$ is also finitely generated by \Cref{thm:dual-cone-fg}. Since $M_n^*=\cone(M_n)^* \cap \ZZ^n$ by \Cref{prop:local-dualproperties}, it follows from \Cref{lem:Gordan} that $M_n^*$ is an affine monoid.  
	\end{proof}
	
	While \Cref{ex:strict-inclusion} shows that the converse of \Cref{prop:local-affine} is false in general, it may hold under additional assumptions. In view of \Cref{thm:dual-cone-fg}, it is natural to ask whether the converse of \Cref{prop:local-affine} holds when $\cone(M_n)$ is closed.

	%-------------------------------------------------------
	\section{Duality for symmetric monoids}
	\label{sec:dualmonoids}
	
	This section explores dual monoids of symmetric monoids in both the local and global settings. Building on the characterizations of positive and non-positive symmetric monoids established in \cite{L25,LRV26}, we analyze how these properties behave under duality. In particular, we provide a complete description of the duals of non-positive symmetric monoids. These results lay the groundwork for the equivariant Minkowski--Weyl theorem established in the next section. 
	
	To begin, note that if $M_n$ is a monoid in $\ZZ^n$, then it can also be viewed as a monoid in $\ZZ^m$ for any $m\ge n$. The definition of $M_n^*$ thus \emph{depends} on the choice of the ambient lattice.
	To avoid ambiguity, for a given chain of monoids $\M=(M_{n})_{n\geq 1}$ we always regard $M_n$ as a subset of $\ZZ^n$ when forming its dual $M_n^*$. For such a chain, we denote by $\M^*=(M_{n}^*)_{n\geq 1}$ the corresponding sequence of dual monoids. It should be noted that $\M^*$ is generally not a $\Sym$-invariant chain; see \Cref{ex:dual-sequence}. Consequently, the framework for $\Sym$-invariant chains of monoids developed in \cite{KLR22,LR23} does not apply directly to $\M^*$.
	Nevertheless, the individual members of the sequence $\M^*$ inherit the symmetry of the original monoids, as shown in the following simple lemma (cf. \cite[Lemma 4.3]{LR23} for the analogous result for cones).
	
	\begin{lemma}
		\label{lem:sym-preservation}
		Assume that $M$ is a monoid in $\ZZ^{(\NN)}$ or $\ZZ^{n}$, and let $\Theta$ be a subgroup of $\Sym$. If $M$ is $\Theta$-invariant, then so is its dual $M^*$.
	\end{lemma}
	
	\begin{proof}
		Let $\ub \in M^*$ and $\sigma \in \Theta$. For any $\vb \in M$, the $\Theta$-invariance of $M$ yields $\sigma^{-1}(\vb) \in M$. It follows that 
		$$\langle\sigma(\ub), \vb\rangle=\sum_{i\ge 1} u_{\sigma^{-1}(i)} v_i=\sum_{j\ge 1} u_j v_{\sigma(j)}=\left\langle\ub, \sigma^{-1}(\vb)\right\rangle \geq 0.$$
		Hence, $\sigma(\ub) \in M^*$, and therefore, $M^*$ is $\Theta$-invariant.
	\end{proof}

	\subsection{Duality for positive symmetric monoids}
	In this subsection, we study how positivity of symmetric monoids behaves under duality. We show that this property is preserved in the global setting. Although an analogous statement does not hold for local monoids in general, we characterize precisely when it fails and explicitly determine the duals in those cases. As an application, we prove that for any non-trivial $\Sym$-invariant chain of positive monoids $\M=(M_{n})_{n\ge 1}$, the corresponding sequence of duals $\M^*=(M_{n}^*)_{n\ge 1}$ is eventually positive. Analogous results for cones are also established.

	Let us first introduce some notation. For a subset $A \subseteq \RR^{(\NN)}$, define
	\begin{align*}
		\Zk(A) &= \{\ub \in A \mid s(\ub)=0\},\\
		\Pk(A) &= \{\ub \in A \mid s(\ub)\ge 0\},\\
		\Nk(A) &= \{\ub \in A \mid s(\ub)\le 0\},
	\end{align*}
	where, as before, $s(\ub)$ denotes the coordinate sum of $\ub$. When  $n \ge 2$, it is shown in \cite[Corollary 5.10]{L25} that 
	\begin{align*}
		\Zk(\ZZ^n) &= \ZZ \,\Sym(n)(\eb_1-\eb_2),\\
		\Pk(\ZZ^n) &= \ZZ^n_{\ge 0} + \ZZ \,\Sym(n)(\eb_1-\eb_2),\\
		\Nk(\ZZ^n) &= \ZZ^n_{\le 0} + \ZZ \,\Sym(n)(\eb_1-\eb_2).
	\end{align*}
	Evidently, all these sets are non-positive submonoids of $\ZZ^n$.
	Analogous descriptions for $\Zk(A)$, $\Pk(A)$, and $\Nk(A)$ when 
	$A\in \{\RR^n,\ZZ^{(\NN)},\RR^{(\NN)}\}$ can be found in \cite[Section 5]{L25}. 
	
	For $n\in \NN$, let $\epb_n=(1,\dots,1)$ denote the all-one vector in $\ZZ^n$. 
	The next result describes the behavior of local positive symmetric monoids under duality.
	
	\begin{proposition}\label{duality: local positive}
		Let $M_n\subseteq \ZZ^{n}$ be a nonzero positive $\Sym(n)$-invariant monoid for some $n\in\NN$. Then:
		\begin{enumerate}
			\item 
			$M_n^*$ is positive if and only if $M_n \nsubseteq \ZZ\epb_n$.
			\item 
			If $M_n\subseteq \ZZ\epb_n$, then either $M_n \subseteq \ZZ_{\ge 0}\epb_n$ or $M_n \subseteq \ZZ_{\le 0}\epb_n$, and we have
			\[
			M_n^*=
			\begin{cases}
				\Pk(\ZZ^n)&\text{if } M_n \subseteq \ZZ_{\ge 0}\epsilon_n,\\
				\Nk(\ZZ^n)&\text{if } M_n \subseteq \ZZ_{\le 0}\epsilon_n.
			\end{cases}
			\]
		\end{enumerate}
	\end{proposition}
	
	\begin{proof}
		We first prove (ii). Assume that $M_n\subseteq \ZZ\epb_n$. Since $M_n$ is positive, \Cref{thm:local positive and non-positive} implies that either 
		$M_n \subseteq \ZZ_{\ge 0}\epb_n$ or $M_n \subseteq \ZZ_{\le 0}\epb_n$. 
		We treat the case $M_n \subseteq \ZZ_{\ge 0}\epb_n$; the other case is analogous.
		Let $\vb\in M_n\setminus\{\nub\}$. Then $\vb=a\epb_n$ for some $a\in \ZZ_{> 0}$. 
		For any $\ub \in \ZZ^n$, we have
		\[
		\langle\ub,\vb \rangle = \langle\ub,a\epb_n\rangle = a\,s(\ub).
		\]
		Hence,
		\[
		M_n^* = \{\ub \in \ZZ^n \mid s(\ub)\ge 0\} = \Pk(\ZZ^n).
		\]
		
		Now (ii) implies that $M_n^*$ is non-positive when $M_n\subseteq \ZZ\epb_n$. 
		To prove (i), it remains to show that $M_n^*$ is positive when 
		$M_n\nsubseteq \ZZ\epb_n$.
		Since $M_n$ is positive, it follows from \Cref{thm:local positive and non-positive} that either $s(\xb)>0$ for all $\xb\in M_n\setminus\{\nub\}$, or $s(\xb)<0$ for all $\xb\in M_n\setminus\{\nub\}$. Without loss of generality, assume the former. Because $M_n \nsubseteq \ZZ\epb_n$, there exists $\vb=(v_1,\dots,v_n) \in M_n$ such that $v_i \ne v_j$ for some $i,j\in [n]$. Using the $\Sym(n)$-invariance of $M_n$, we may assume 
		$v_1< v_2\le \cdots \le v_n$. By \Cref{thm:local positive and non-positive}, it suffices to show that $s(\ub)>0$ for every $\ub\in M_n^*\setminus \{\nub \}$. Suppose, to the contrary, that there exists $\ub=(u_1,\dots,u_n) \in M_n^*\setminus \{\nub \}$ with $s(\ub)\le 0$. Since $M_n^*$ is also  $\Sym(n)$-invariant by \Cref{lem:sym-preservation}, we can assume that $u_1\ge u_2\ge\dots\ge u_n$. Then the rearrangement inequality (see, e.g., \cite[Theorem 368]{HLP}) yields 
		\[
		\langle\sigma(\ub),\vb\rangle\ge \langle\ub,\vb\rangle \ge 0 
		\quad \text{for all } \sigma \in \Sym(n).
		\]
		On the other hand, since $s(\vb)>0\ge s(\ub)$, we have 
		\begin{align*}
			\sum_{\sigma\in \Sym(n)}\langle\sigma(\ub),\vb\rangle
			&=\sum_{\sigma\in \Sym(n)}\sum_{i=1}^nu_{\sigma^{-1}(i)}v_i
			=\sum_{i=1}^nv_i\sum_{\sigma\in \Sym(n)}u_{\sigma^{-1}(i)}\\
			&=\sum_{i=1}^nv_i(n-1)!s(\ub)=(n-1)!s(\ub)s(\vb)\le 0.
		\end{align*}
		It follows that $s(\ub)=0$ and
		\begin{equation}
			\label{eq:null-dot-product}
			\langle\sigma(\ub),\vb\rangle =0
			\quad \text{for all } \sigma \in \Sym(n).
		\end{equation}
		Since $\ub\ne \nub$ and $s(\ub)=0$, there exists $k,l\in [n]$ such that $u_k>0>u_l$. Choose $\tau\in\Sym(n)$ with $\tau^{-1}(1)=k$, $\tau^{-1}(2)=l$, and let $\pi$ be the transposition swapping $1$ and $2$. Then
		\[
		\langle\tau(\ub),\vb\rangle-\langle\pi\tau(\ub),\vb\rangle=u_kv_1+u_lv_2-u_lv_1-u_kv_2=(u_k-u_l)(v_1-v_2) \ne 0.
		\]
		This contradicts \eqref{eq:null-dot-product}. Thus, $s(\ub) >0$ for all $\ub\in M_n^*\setminus\{\nub\}$, and the proof is complete.
	\end{proof}

	Applying \Cref{duality: local positive} to $\Sym$-invariant chains yields the following consequence. We call a chain \emph{non-trivial} if it contains a nonzero member.
	
	\begin{corollary}\label{duality:chain of positive monoids}
		Let $\M=(M_{n})_{n\geq 1}$ be a non-trivial $\Sym$-invariant chain of positive monoids. Then the sequence of dual monoids $\M^*=(M_{n}^*)_{n\geq 1}$ is eventually positive. More precisely, $M_n^*$ is positive for all $n>r$, where $r=\min\{n\mid M_n\ne\{\nub\}\}$.
	\end{corollary}
	
	\begin{proof}
		Let $\ub \in M_r\setminus\{\nub\}$. Then $\ub\in M_n$ for all $n>r$, and clearly $\ub \notin \ZZ\epsilon_n$. Hence $M_n \nsubseteq \ZZ\epsilon_n$, and the claim follows from \Cref{duality: local positive}.
	\end{proof}
	
	In the global setting, positivity behaves better under duality. Let $\epb$ denote the all-one vector in $\ZZ^\NN$. Since $\ZZ^{(\NN)} \cap \ZZ\epb = \{\nub\}$, we have $M \nsubseteq \ZZ\epb$ for any nonzero monoid $M\subseteq\ZZ^{(\NN)}$. 
	Arguing as above, we obtain the following.

	\begin{proposition}\label{duality: global positive}
		Let $M\subseteq \ZZ^{(\NN)}$ be a nonzero positive $\Sym$-invariant monoid. Then $M^*$ is also positive.
	\end{proposition}
	
	The property of being a positive monoid is intrinsically linked to the geometry of the cone it generates. Recall that a cone $C\subseteq \RR^{(\NN)}$ is \emph{pointed} if its \emph{lineality space}
	\[
	\lin(C)=\{\ub\in C\mid -\ub\in C\}
	\]
	is trivial. The next lemma, proved in \cite[Lemma 5.8]{L25}, slightly generalizes a classical result for affine monoids (see, e.g., \cite[Proposition 2.16]{BG09}). 
	
	\begin{lemma}
		\label{l:positive-pointed}
		A monoid $M\subseteq \ZZ^{(\NN)}$ is positive if and only if $\cone(M)$ is pointed.
	\end{lemma}
	
	By using arguments analogous to the monoid case, we obtain the following results for symmetric cones.

	\begin{proposition}
		Let $C_n\subseteq \RR^{n}$ be a nonzero pointed $\Sym(n)$-invariant cone. Then:
		\begin{enumerate}
			\item  
			$C_n^*$ is pointed if and only if $C_n \nsubseteq \RR\epb_n$.
			\item
			If $C_n \subseteq \RR\epb_n$, then either $C_n = \RR_{\ge 0}\epb_n$ or $C_n = \RR_{\le 0}\epb_n$, and we have
			\[
			C_n^*=
			\begin{cases}
				\Pk(\RR^n)&\text{if } C_n =\RR_{\ge 0}\epb_n,\\
				\Nk(\RR^n)&\text{if } C_n =\RR_{\le 0}\epb_n.
			\end{cases}
			\]
		\end{enumerate}
	\end{proposition}
	
	\begin{corollary}
		Let $\C=(C_{n})_{n\geq 1}$ be any non-trivial $\Sym$-invariant chain of pointed cones. Then the sequence of dual cones $\C^*=(C_{n}^*)_{n\geq 1}$ is eventually pointed. More precisely, $C_n^*$ is pointed for all $n>r$, where $r=\min\{n\mid C_n\ne\{\nub\}\}$.
	\end{corollary}
	
	\begin{proposition}
		Let $C\subseteq \RR^{(\NN)}$ be a nonzero pointed $\Sym$-invariant cone. Then $C^*$ is also pointed.
	\end{proposition}

	\subsection{Duality for non-positive symmetric monoids} \label{subsec:non-positive dual}
	We now provide explicit descriptions of duals of non-positive symmetric monoids. In the local setting, we obtain the following classification.
	
	\begin{proposition}\label{duality: local non-positive}
		Let $M_n\subseteq \ZZ^{n}$ be a non-positive $\Sym(n)$-invariant monoid for some $n\ge 2$. Then $M_n^*$ is one of the following:
		\begin{align*}
			\mathfrak{N}_1^* = \{\nub\}, \quad
			\mathfrak{N}_2^* = \ZZ\epb_n, \quad
			\mathfrak{N}_3^* = \ZZ_{\ge 0}\epb_n,\quad
			\mathfrak{N}_4^* = \ZZ_{\le 0}\epb_n,\quad
			\mathfrak{N}_5^* = \Zk(\ZZ^n). 
		\end{align*}
	\end{proposition}

	The proof of this result relies on the classification of non-pointed local symmetric cones established in \cite[Proposition 5.5]{L25}:
	
	\begin{proposition}
		\label{p:non-pointed-local-cones}
		Let $C_n\subseteq \RR^{n}$ be non-pointed $\Sym(n)$-invariant cone for some $n\ge 2$. Then $C_n$ is one of the following:
		\begin{align*}
			\mathfrak{D}_1=\RR^{n},\quad
			\mathfrak{D}_2=\Zk(\RR^{n}),\quad
			\mathfrak{D}_3=\Pk(\RR^{n}),\quad
			\mathfrak{D}_4=\Nk(\RR^{n}),\quad
			\mathfrak{D}_5=\RR\epb_n.
		\end{align*}
	\end{proposition}
	
	This classification immediately yields the corresponding description of dual cones.
	
	\begin{corollary}
		\label{p:dual-non-pointed-local-cones}
		Let $C_n\subseteq \RR^{n}$ be non-pointed $\Sym(n)$-invariant cone for some $n\ge 2$. Then $C_n^*$ is one of the following:
		\begin{align*}
			\mathfrak{D}_1^*=\{\nub\},\quad
			\mathfrak{D}_2^*=\RR\epb_n,\quad
			\mathfrak{D}_3^*=\RR_{\ge 0}\epb_n,\quad
			\mathfrak{D}_4^*=\RR_{\le 0}\epb_n,\quad
			\mathfrak{D}_5^*=\Zk(\RR^n).
		\end{align*}
	\end{corollary}

	\begin{proof}[Proof of \Cref{duality: local non-positive}]
		By \Cref{l:positive-pointed}, $\cone(M_n)$ is non-pointed. Hence, $\cone(M_n)^*$ must be one of the cones $\mathfrak{D}_i^*$ listed in \Cref{p:dual-non-pointed-local-cones}. Since $\mathfrak{N}_i^*=\mathfrak{D}_i^*\cap\ZZ^n$ for all $i\in[5]$, the result follows from \Cref{prop:local-dualproperties}.
	\end{proof}
	
	\begin{remark}
		\label{rm:n=1}
		When $n=1$, $\RR$ is the unique non-pointed cone in $\RR$. Similarly, $\ZZ$ is the unique non-positive monoid in $\ZZ$. Thus, the only dual of a non-positive monoid in $\ZZ$ is $\{0\}$.
	\end{remark}
	
	\Cref{duality: local non-positive} provides a simple way to construct a $\Sym$-invariant chain of monoids whose dual sequence fails to be $\Sym$-invariant.
	
	\begin{example}
		\label{ex:dual-sequence}
		Let $M_n=\Zk(\ZZ^n)$ for $n\ge 1$. Then $\M=(M_{n})_{n\geq 1}$ is clearly a $\Sym$-invariant chain. By \Cref{duality: local non-positive} (and \Cref{rm:n=1}), we have
		$
		M_n^*=\ZZ\epb_n
		$ for $n\ge 1$. Thus, the sequence $\M^*=(M_{n}^*)_{n\geq 1}$ is not even an ascending chain.
	\end{example}
	
	\Cref{p:dual-non-pointed-local-cones} has the following global analogue, which is derived from the classification of non-pointed $\Sym$-invariant cones in \cite[Theorem 5.1]{L25}.
	
	\begin{proposition}
		\label{p:dual-non-pointed-global-cones}
		Let $C\subseteq \RR^{(\NN)}$ be non-pointed $\Sym$-invariant cone. Then $C^*$ is one of the following:
		\begin{align*}
			\mathfrak{C}_1^*=\{\nub\},\quad
			\mathfrak{C}_2^*=\RR\epb,\quad
			\mathfrak{C}_3^*=\RR_{\ge 0}\epb,\quad
			\mathfrak{C}_4^*=\RR_{\le 0}\epb.
		\end{align*}
	\end{proposition}
	
	Combining this with \Cref{prop:global-dualproperties,l:positive-pointed}, we obtain the global analogue of \Cref{duality: local non-positive}.
	
	\begin{proposition}\label{duality: global non-positive}
		Let $M\subseteq \ZZ^{(\NN)}$ be a non-positive $\Sym$-invariant monoid. Then $M^*$ is one of the following:
		\begin{align*}
			\mathfrak{M}_1^* = \{\nub\}, \quad
			\mathfrak{M}_2^* = \ZZ\epb, \quad
			\mathfrak{M}_3^* = \ZZ_{\ge 0}\epb,\quad
			\mathfrak{M}_4^* = \ZZ_{\le 0}\epb. 
		\end{align*}
	\end{proposition}

	%-------------------------------------------------------
	\section{Equivariant Minkowski--Weyl theorem}
	\label{sec:M-W}
	
	As discussed in \Cref{sec:dual}, the classical Minkowski--Weyl theorem (\Cref{thm:clasical-Minkowski-Weyl}) can be interpreted in the language of duality: the dual of a finitely generated cone in $\RR^n$ is again finitely generated. Moreover, the generators of the dual cone are determined by the linear half-space description of the original cone (see \Cref{thm:dual-cone-fg}). 
	Extensions of this duality to the equivariant setting have recently been developed in \cite{L25, LR23}. Building on these developments, we establish in this section equivariant analogues of the Minkowski--Weyl theorem for monoids, treating both the global and local settings.

	\subsection{The global setting}
	As noted earlier, finite generation of a global monoid $M \subseteq \ZZ^{(\NN)}$ is not preserved under duality in general. The following result shows that this failure persists even when symmetry is imposed. An analogous phenomenon for cones was observed in \cite[Theorem 4.9]{LR23}.
	
	\begin{proposition}\label{dualmonoid:not equi f.g}
		For any $\Sym$-invariant monoid $M \subseteq \ZZ_{\ge 0}^{(\NN)}$, its dual monoid $M^*\subseteq\ZZ^\NN$ is not $\Sym$-equivariantly finitely generated.
	\end{proposition}
	
	\begin{proof}
		Suppose, for contradiction, that $M^*$ is $\Sym$-equivariantly generated by a finite set $A \subseteq \ZZ^\NN$. Then
		$\gp(M^*)=\gp(\Sym(A)).$
		Since $\Sym$ is countable and $A$ is finite, $\gp(M^*)$ is a countable set. On the other hand, the assumption $M \subseteq \ZZ_{\ge 0}^{(\NN)}$ implies that $\ZZ_{\ge 0}^{\NN} \subseteq M^*$. Consequently, 
		$$\gp(M^*)=\gp(\ZZ_{\ge 0}^{\NN})=\ZZ^\NN.$$
		Since $\ZZ^\NN$ is uncountable, we reach a contradiction. Thus, $M^*$ cannot be $\Sym$-equivariantly finitely generated.
	\end{proof}
	
	This result shows that a direct equivariant analogue of the Minkowski--Weyl theorem fails for global monoids. Nevertheless, by considering only elements with finite support, we obtain a complete description of the restricted dual monoid. This mirrors the behavior of global $\Sym$-invariant cones described in \cite[Proposition 3.1]{L25}.
	
	\begin{proposition}\label{M-W:global}
		Let $M\subseteq \ZZ^{(\NN)}$ be a $\Sym$-invariant monoid. Then the restricted dual monoid $M^* \cap \ZZ^{(\NN)}$ is one of the following: 
		$$ \{\nub\}, \quad \ZZ_{\ge 0}^{(\NN)}, \quad \ZZ_{\le 0}^{(\NN)}, \quad \ZZ^{(\NN)}.$$
		In particular, $M^* \cap \ZZ^{(\NN)}$ is always $\Sym$-equivariantly finitely generated.
	\end{proposition}
	\begin{proof}
		The monoids $\ZZ_{\ge 0}^{(\NN)}$, $\ZZ_{\le 0}^{(\NN)}$, and $\ZZ^{(\NN)}$ are clearly $\Sym$-equivariantly finitely generated by $\eb_1$, $-\eb_1$, and $\{\eb_1,-\eb_1\}$, respectively. Thus, it suffices to establish the classification. Note that $\cone(M)\subseteq \RR^{(\NN)} $ is a $\Sym$-invariant cone. By \cite[Proposition 3.1]{L25}, $\cone(M)^*\cap \RR^{(\NN)}$ must be one of the following:
		$$ \{\nub\}, \quad \RR_{\ge 0}^{(\NN)}, \quad \RR_{\le 0}^{(\NN)}, \quad \RR^{(\NN)}.$$
		On the other hand, by \Cref{prop:global-dualproperties}, we have
		$$M^*=\cone(M)^* \cap \ZZ^{\NN}.$$
		Therefore, 
		$$M^* \cap \ZZ^{(\NN)}=\cone(M)^* \cap \ZZ^{(\NN)}=\cone(M)^*\cap \RR^{(\NN)}\cap \ZZ^{(\NN)},$$ 
		which yields the desired classification. 
	\end{proof}
	
	\subsection{The local setting}
	Given a $\Sym$-invariant monoid $M \subseteq \ZZ^{(\NN)}$, \Cref{dualmonoid:not equi f.g} indicates that a direct analysis of its dual $M^*$ is generally intractable. A more effective approach is to pass to the local setting. When $M$ is $\Sym$-equivariantly finitely generated, it can be encoded by a stabilizing chain of local monoids $\M = (M_n)_{n \ge 1}$, which is eventually affine (see \cite[Theorem 5.1]{LRV26}). To understand $M^*$, it is therefore natural to investigate the sequence of dual monoids $\M^* = (M_n^*)_{n \ge 1}$. This perspective is inspired by the study of sequences of dual cones in \cite[Theorem 3.3]{L25}. In this subsection, we develop an analogous framework for monoids and establish a local equivariant Minkowski--Weyl theorem, which provides an explicit description of how the generators of $M_n^*$ evolve as $n$ increases.
	
	We first observe that any symmetric monoid can be equivariantly generated by ordered elements. To formalize this, we introduce some notation.
	For $n\in \NN$, define
	\begin{align*}
		\OO(\ZZ^n)&=\{(u_1,\dots,u_{n})\in \ZZ^{n}\mid u_1\le\dots\le u_{n}\},\\
		\OO^-(\ZZ^n)&=\{(v_1,\dots,v_{n})\in \ZZ^{n}\mid v_1\ge\dots\ge v_{n}\}.
	\end{align*}
	More generally, for any subset $D\subseteq \ZZ^n$, we set $\OO(D)=\OO(\ZZ^n)\cap D.$ 
	
	\begin{lemma}
		\label{prop:ordered-monoid}
		Let $M_n\subseteq\ZZ^n$ be a monoid. Then $\OO(M_n)$ is a submonoid of $M_n$. Moreover:
		\begin{enumerate}
			\item If $M_n$ is affine, then $\OO(M_n)$ is also affine.
			\item If $M_n$ is $\Sym(n)$-invariant, then any generating set of $\OO(M_n)$ is a $\Sym(n)$-equivariant generating set of $M_n$.
		\end{enumerate}
	\end{lemma}
	
	\begin{proof}
		It is evident that $\OO(M_n)$ is a submonoid of $M_n$. 
		
		(i) By definition,
		\[
		\OO(M_{n})=M_n \cap \bigcap_{i=1}^{n-1} H_i^{+},
		\quad\text{where }\ H_i^{+}=\left\{\left(x_1, \ldots, x_n\right) \in \ZZ^n\mid x_i \leq x_{i+1}\right\}.
		\]
		Since $\bigcap_{i=1}^{n-1} H_i^{+}$ is a finitely generated cone by \Cref{thm:clasical-Minkowski-Weyl}, it follows from \Cref{lem:Gordan} that $\OO(M_{n})$ is an affine monoid if $M_n$ is so.
		
		(ii) If $M_n$ is $\Sym(n)$-invariant, then $M_n=\Sym(n)(\OO(M_n))$. Hence any generating set of $\OO(M_n)$ is a $\Sym(n)$-equivariant generating set of $M_n$.
		\end{proof}
	
	As illustrated in \Cref{example:M-W}, the converse of \Cref{prop:ordered-monoid}(ii) does not hold in general: a $\Sym(n)$-equivariant generating set of $M_n$ need not generate $\OO(M_n)$.
	
	To state our main result, we introduce further notation. For integers $a\le b$, set 
	\[
	J_n(a,b)=\OO([a,b]^n\cap \ZZ^n)
	=\{(u_1,\dots,u_{n})\in \ZZ^{n}\mid a\le u_1\le\dots\le u_{n}\le b\}.
	\]
	Let $\ub=(u_1,\dots,u_n)\in\OO(\ZZ^n)$ and fix an index $i\in[n-1]$. For $m>n$, define the \emph{insertion operator}:
	\[
	\OO_{i,m}(\ub)=
		\{(u_1,\dots,u_{i})\}\times J_{m-n}(u_i,u_{i+1})
		\times\{(u_{i+1},\dots,u_n)\}
		\subseteq \OO(\ZZ^m).
	\] 
	Thus, $\OO_{i,m}(\ub)$ consists of all vectors obtained from $\ub$ by inserting $m-n$ additional entries between the $i$-th and $(i+1)$-st coordinates, while preserving the nondecreasing order. In other words, this set can be defined inductively as follows:
	\[
	\OO_{i,m}(\ub)=\bigcup_{\vb\in\OO_{i,m-1}(\ub)}\OO_{i,m}(\vb).
	\]
	For example, if $\ub=(1,2,4,5)\in \OO(\ZZ^{4})$, then
	\begin{align*}
		\OO_{2,5}(\ub)=\{&(1,2,2,4,5),\, (1,2,3,4,5),\,(1,2,4,4,5)\},\\
		\OO_{2,6}(\ub))=\bigcup_{\vb\in\OO_{2,5}(\ub)}\OO_{2,6}(\vb)
		=\{&(1,2,2,2,4,5),\, (1,2,2,3,4,5),\,(1,2,3,3,4,5),\\
		&(1,2,2,4,4,5),\, (1,2,3,4,4,5),\, (1,2,4,4,4,5)\}.
	\end{align*}
	
	The signs of the coordinates of the original monoid generators play a crucial role in the description of dual monoids. For a subset $A \subseteq \ZZ^n$, define
	\[
	\mu_+(A)=\max_{\ub\in A}|\{i\mid u_i>0\}|
	\quad\text{and}\quad
	\mu_-(A)=\max_{\ub\in A}|\{i\mid u_i<0\}|.
	\]
	For instance, if
	\[
	A=\{(1,0,-3,-4),\, (2,3,4,-1)\},
	\]
	then $\mu_+(A)=3$ and $\mu_-(A)=2$.
	
	We are now ready to state the main result of this section.
	
	\begin{theorem} \label{M-W:local general}
		Let $\M=(M_{n})_{n\geq 1}$ be a stabilizing $\Sym$-invariant chain of monoids with $\ind(\M)=r$, and let $G_r$ be a $\Sym(r)$-equivariant generating set for $M_r$. Choose positive integers $s<t$ such that $s \ge \mu_+(G_r)$ and $t\ge \max \{s+\mu_-(G_r),r\}$. Suppose $E_{t}\subseteq \OO(\ZZ^{t})$ is a generating set for $\OO(M_{t}^*)$. Then for all $n>t$, the monoid $\OO(M_{n}^*)$ is generated by
		\[  
		E_n\defas\bigcup_{\ub\in E_{t}} \OO_{s,n}(\ub).
		\]
		In particular, $E_n$ is a $\Sym(n)$-equivariant generating set for $M_{n}^*$ for all $n>t$.
	\end{theorem}
	
	\begin{remark} 
		\label{remark:on M-W local}
		Keep the notation of \Cref{M-W:local general}.
		\begin{enumerate}
			\item 
			By \Cref{prop:ordered-monoid}(ii), the assumption that $E_{t}$ is a generating set for $\OO(M_{t}^*)$ is stronger than requiring it to be a $\Sym(t)$-equivariant generating set for $M_{t}^*$. One may ask whether the latter condition alone suffices to ensure that $E_n$ is a $\Sym(n)$-equivariant generating set for $M_{n}^*$ for all $n>t$. In general, this is false, as shown in \Cref{example:M-W}. However, the statement does hold if $t$ is chosen sufficiently large (see \Cref{M-W:nonnegative,M-W:local 2r}).
			
			\item 
			If the chain $\M$ consists of affine monoids, then by \Cref{prop:local-affine}, the dual monoids $M_n^*$ are also affine. Consequently, the submonoids $\OO(M_{n}^*)$ are affine as well by \Cref{prop:ordered-monoid}(i). Thus, if one chooses a finite generating set $E_t$ for $\OO(M_{t}^*)$, then \Cref{M-W:local general} yields finite $\Sym(n)$-equivariant generating sets for $M_n^*$ for all $n>t$.
			\end{enumerate}
	\end{remark}
	
	A convenient specialization of \Cref{M-W:local general} is obtained by choosing $s$ and $t$ based solely on the stability index. The next result is analogous to the equivariant Minkowski--Weyl theorem for cones established in \cite[Theorem 3.3]{L25}.
	
	\begin{corollary}
		\label{cor:M-W-simplified}
		Let $\M=(M_{n})_{n\geq 1}$ be a stabilizing $\Sym$-invariant chain of monoids with $\ind(\M)=r$. Suppose $E_{2r}\subseteq \OO(\ZZ^{2r})$ is a generating set for $\OO(M_{2r}^*)$. Then for all $n>2r$,
		\[  
		E_n\defas\bigcup_{\ub\in E_{2r}} \OO_{r,n}(\ub)
		\]
		is a generating set for $\OO(M_{n}^*)$, and hence a $\Sym(n)$-equivariant generating set for $M_{n}^*$.
	\end{corollary}
	
	\begin{proof}
		Let $G_r$ be a $\Sym(r)$-equivariant generating set for $M_r$. Then $\mu_+(G_r), \mu_-(G_r) \le r$, so we may take $s=r$ and $t=2r$ in \Cref{M-W:local general}.
	\end{proof}
	
	For nonnegative monoids, \Cref{M-W:local general} can be refined as follows. An analogous result for nonnegative cones appears in \cite[Theorem 4.4]{LR23}.

	\begin{corollary}
		\label{cor:M-W-nonnegative}
		Let $\M=(M_{n})_{n\geq 1}$ be a stabilizing $\Sym$-invariant chain of monoids with $\ind(\M)=r$, and assume that $M_n \subseteq \ZZ^n_{\ge 0}$ for all $n\ge 1$. Suppose $E_{r+1}$ is a generating set for $\OO(M_{r+1}^*)$. Then for all $n>r+1$,
		$$E_n=\bigcup_{\ub\in E_{r+1}} \OO_{r,n}(\ub)$$ 
		is a generating set for $\OO(M_{n}^*)$, and hence a $\Sym(n)$-equivariant generating set for $M_{n}^*$.
	\end{corollary}
	
	\begin{proof}
		For any $\Sym(r)$-equivariant generating set $G_r$ of $M_r$, we have $\mu_-(G_r)=0$ because $G_r\subseteq \ZZ^r_{\ge 0}$. Morover, it always holds that  $\mu_+(G_r) \le r$. Thus, we can  choose $s=r$ and $t=r+1$ in \Cref{M-W:local general}.
	\end{proof}
	
	Before turning to the proof of \Cref{M-W:local general}, we illustrate the result and its consequences with an example.
	
	\begin{example}\label{example:M-W}
		Let $\M = (M_n)_{n \ge 1}$ be a $\Sym$-invariant chain with $\ind(\M) = 3$.
		\begin{enumerate}[leftmargin=*]
			\item 
			Suppose $M_3$ is $\Sym(3)$-equivariantly generated by 
			\[
			G_3=\{(-1,-1,3),\,(-1,1,1)\}.
			\]
			Then $\mu_+(G_3) = \mu_-(G_3) = 2$. Taking $s=2$ and $t=4$, we compute (using Normaliz \cite{BISV} and Macaulay2 \cite{GS}) a generating set for $\OO(M_4^*)$:
			\[
			E_4=\{(1^{(4)}),\,(1^{(3)},2),\,(2^{(2)},3^{(2)}),\,(2,3^{(3)})\},
			\]
			where $a^{(k)}$ means that the entry $a$ is repeated $k$ times.
			By \Cref{M-W:local general}, for all $n>4$, the monoid $M_n^*$ is   $\Sym(n)$-equivariantly generated by
			$$E_n=\bigcup_{\ub\in E_{4}} \OO_{2,n}(\ub)
			=\{(1^{(n)}),\,(1^{(n-1)},2)\}\cup\{(2^{(k)},3^{(n-k)})\mid k\in[n-2]\}.$$  
			Alternatively, applying \Cref{cor:M-W-simplified}, one may compute a generating set $E_6'$ for $\OO(M_6^*)$. Then for all $n > 6$, the set  
			$$E_n'=\bigcup_{\ub\in E_{6}'} \OO_{3,n}(\ub)$$  
			forms a $\Sym(n)$-equivariant generating set for $M_{n}^*$. Observe that if one takes $E_{6}'= E_{6}$, then $E_{n}'= E_{n}$ for all $n>6$.
			
			On the other hand, consider the $\Sym(4)$-equivariant Hilbert basis
			\[
			H_4=\{(1^{(4)}),\,(1^{(3)},2),\,(2,3^{(3)})\}
			\]
			for $M_{4}^*$. Then for all $n>4$,
			\[
			H_n\defas\bigcup_{\ub\in H_{4}} \OO_{2,n}(\ub)
			=\{(1^{(n)}),\,(1^{(n-1)},2),\,(2,3^{(n-1)})\}
			\]
			is not a $\Sym(n)$-equivariant generating set for $M_{n}^*$. Indeed, one checks using coordinate sums that $(2^{(2)},3^{(n-2)})\in M_{n}^*$ does not lie in the monoid generated by $\Sym(n)(H_n)$. This shows that $H_n$ does not generate $\OO(M_n^*)$, and that \Cref{M-W:local general} fails if one only assumes that $E_t$ is a $\Sym(t)$-equivariant generating set for $M_t^*$.
			
			\item 
			Suppose instead that $M_3$ is $\Sym(3)$-equivariantly generated by
			\[
			G_3=\{(1,1,2),\, (1,2,3)\} \subseteq \ZZ^3_{\ge 0}.
			\]
			Taking $s=3$ and $t=4$, we find a generating set of $\OO(M_4^*)$:
			\[
			E_4=\Big\{\sum_{i=k}^4 \eb_i \mid k\in[4]\Big\} \cup\Big\{(-1^{(2)},5^{(2)}),\,(-1,0,3^{(2)}),\,(-1,1^{(3)})\Big\}.
			\]
			By \Cref{cor:M-W-nonnegative}, for all $n>4$, the monoid $M_n^*$ is   $\Sym(n)$-equivariantly generated by
			\begin{align*}
				E_n&=\bigcup_{\ub\in E_{4}} \OO_{3,n}(\ub)\\
				&=\Big\{\sum_{i=k}^n \eb_i \mid k\in[n]\Big\} \cup\Big\{(-1^{(2)},5^{(n-2)}),\,(-1,0,3^{(n-2)}),\,(-1,1^{(n-1)})\Big\}.
			\end{align*}
		\end{enumerate} 
	\end{example}
	
	Let us now proceed to the proof of \Cref{M-W:local general}. Although its statement parallels the equivariant Minkowski--Weyl theorem for cones in \cite[Theorem 3.3]{L25}, the structure of the generating sets in the monoid setting differs, necessitating a distinct argument. In fact, while the inclusion $\mndcl(E_n)\subseteq \OO(M_n^*)$ can be established by adapting the strategy of \cite[Theorem 3.3]{L25}, proving the reverse inclusion requires additional ideas.
	
	We first establish several auxiliary results. The following lemma is a monoid analogue of \cite[Lemma 3.10]{L25}.
	
	\begin{lemma}
		\label{lem:sucessive dual}
		Let $\M=(M_n)_{n\ge 1}$ be a $\Sym$-invariant chain of monoids. Then for any element $\ub=(u_1,\dots,u_{n+1})\in M_{n+1}^*$ and any index $i\in[n+1]$, we have
		\[
		\hat{\ub}=(u_1,\dots,u_{i-1},u_{i+1},\dots,u_{n+1})\in M_n^*.
		\]
	\end{lemma}
	
	\begin{proof}
		Let $\C=(C_n)_{n\ge 1}$ be the associated chain of cones, where $C_n=\cone(M_n)$ for $n\ge 1$. Since $\ub\in M_{n+1}^*\subseteq C_{n+1}^*$, it follows from \cite[Lemma 3.10]{L25} that $\hat{\ub}\in C_n^*$. Now applying \Cref{prop:local-dualproperties}, we obtain 
		$
		\hat{\ub}\in C_n^*\cap\ZZ^n= M_n^*.
		$
	\end{proof}
	
	The next result establishes the inclusion $\mndcl(E_n)\subseteq \OO(M_n^*)$. For its proof, we need some additional notation.
	For $\vb=(v_1,\dots,v_n)\in \ZZ^n$, let $\vb^{-}\in \OO^-(\ZZ^n)$ denote the vector obtained by sorting the coordinates of $\vb$ in non-increasing order. If $m\ge n$ and we regard $\vb$ as an element $\tilde{\vb}\in \ZZ^m$ (by appending zeros), then we write $\vb^{-}(m)$ for the vector $\tilde{\vb}^-$. That is, $\vb^{-}(m)$ is obtained from $\vb^{-}$ by inserting $m-n$ zeros between its positive and negative entries. For example, if $\vb=(1,-2,0,3)$, then $\vb^{-}=(3,1,0,-2)$ and $\vb^{-}(6)=(3,1,0,0,0,-2)$.

	\begin{lemma}
		\label{lem:OO(u)}
		Under the assumptions of \Cref{M-W:local general}, we have
		\[
		\OO_{s,n}(\ub) \subseteq \OO(M_{n}^*)
		\quad\text{for all } \ub\in \OO(M_t^*) \text{ and } n>t.
		\]
	\end{lemma}
	\begin{proof}
		Assume that $\ub=(u_1,\dots,u_t)\in \OO(M_t^*)$. Let $\ub'\in \OO_{s,n}(\ub)$ for some $n>t$. We need to show that $\ub'\in M_{n}^*$.  For $m\ge r$, we may regard $G_r$ as a subset of $\ZZ^m$. Define
		$$G_m^-=\{\vb^- (m)\mid \vb \in G_r\} \subseteq \OO^-(\ZZ^m).$$
		Evidently, $G_m^-=\tau(G_r)$ for some $\tau\in \Sym(m)$.
		Since $\ind(\M)=r$ and $G_r$ is a $\Sym(r)$-equivariant generating set for $M_r$, it follows that $G_m^-$ is a $\Sym(m)$-equivariant generating set for $M_m$. For any $\vb'\in G_n^-$ and any $\sigma \in \Sym(n)$, the rearrangement inequality (see \cite[Theorem 368]{HLP})  implies that
		\[
		\langle \ub',\sigma(\vb')\rangle \ge \langle \ub',\vb'\rangle.
		\]
		Therefore, to prove $\ub'\in M_{n}^*$, it suffices to show that 
		$$\langle \ub',\vb'\rangle \ge 0
		\quad\text{for all } \vb'\in G_n^-.$$ 
		Since $\ub'\in \OO_{s,n}(\ub)$, there exists $\ab \in J_{n-t}(u_s,u_{s+1})$ such that
		$$\ub'=(u_1,\dots,u_s,\ab,u_{s+1},\dots,u_t).$$
		By construction of $G_n^-$, each $\vb'=(v'_1,\dots,v'_n)\in G_n^-$ satisfies
		$$v'_i=0  \quad \text{ for all } \quad  \mu_+(G_r)+1\le i\le n-\mu_-(G_r).$$
		Since $\mu_+(G_r)+1\le s+1\le s+(n-t)\le n-\mu_-(G_r)$ by the assumption on $s$ and $t$, we have
		$$v'_i=0 \quad \text{for all} \quad s+1 \le i \le s+(n-t).$$ 
		Removing these zero coordinates yields a vector $\vb \in G_t^-\subseteq M_{t}$. Since $\ub\in M_t^*$, we obtain
		\begin{align*}
			\langle \ub',\vb'\rangle  = \langle \ub,\vb\rangle \ge 0,
		\end{align*} 
		as required.
	\end{proof}
	
	We next record a simple yet useful fact, which plays a key role in proving the reverse inclusion $\OO(M_n^*)\subseteq\mndcl(E_n) $. This step distinguishes the proof of \Cref{M-W:local general} from that of \cite[Theorem 3.3]{L25}.
	
	\begin{lemma}
		\label{lem:ineq}
		Let $\ub=(u_1,\dots,u_n),\,\vb=(v_1,\dots,v_n)\in\ZZ^n$ satisfy $u_i\le v_i$ for all $i\in [n]$. Then for any $k \in \ZZ$ with $s(\ub)\le k\le s(\vb)$, there exists $\wb=(w_1,\dots,w_n)\in\ZZ^n$ such that $s(\wb)=k$ and $u_i\le w_i \le v_i$ for all $i\in [n]$.
	\end{lemma}
	
	\begin{proof}
		Let $d=k-s(\ub)$. Then
		\[
		0\le d\le s(\vb)-s(\ub)=\sum_{i=1}^n(v_i-u_i).
		\]
		Hence, we can choose $\zb=(z_1,\dots,z_n)\in\ZZ^n$ such that $0\le z_i\le v_i-u_i$ for all $i\in [n]$ and $s(\zb)=d$.
		Setting $\wb=\ub+\zb$ yields the desired vector.
	\end{proof}
	
	We are now in the position to prove \Cref{M-W:local general}.
	
	\begin{proof}[Proof of \Cref{M-W:local general}]
		By \Cref{prop:ordered-monoid}(ii), it suffices to show that $E_n$ generates $\OO(M_{n}^*)$ for all $n>t$. Applying \Cref{lem:OO(u)}, we obtain
		$$E_n= \bigcup_{\ub\in E_{t}} \OO_{s,n}(\ub)\subseteq \OO(M_n^*)
		\quad\text{for all } n>t.$$
		Hence, $\mndcl(E_n) \subseteq \OO(M_n^*)$. For the reverse inclusion, 
		we first consider the case $n=t+1$. Let $\xb=(x_1,\dots,x_{t+1})\in \OO(M_{t+1}^*).$ By \Cref{lem:sucessive dual}, removing the $(s+1)$-st coordinate yields
		$$\hat\xb= (x_1,\dots,x_s,x_{s+2},\dots,x_{t+1})\in \OO(M_{t}^*)=\mndcl(E_{t}).$$ 
		Thus, we can write 
		$$\hat\xb=\yb_{1}+\dots+\yb_{m},$$ 
		where $\yb_{i}=(y_{i,1},\dots, y_{i,t})\in E_{t}$ for all $i\in [m]$. We have
		$$x_s=y_{1,s}+\dots+y_{m,s} \le x_{s+1} \le y_{1,s+1}+\dots+y_{m,s+1}=x_{s+2}.$$
		By \Cref{lem:ineq}, there exists $(z_1,\dots,z_m)\in\ZZ^m$ such that
		$$x_{s+1}=\sum_{i=1}^mz_i, \quad \text{and} \quad y_{i,s}\le z_i \le y_{i,s+1} \text{ for all } i\in[m].$$ 
		For $i\in[m]$, let $\zb_i=(y_{i,1},\dots,y_{i,s},z_i,y_{i,s+1},\dots, y_{i,t})$. Then $\zb_i\in E_{t+1}$, showing that
		$$\xb=\zb_1+ \cdots+\zb_m\in \mndcl(E_{t+1}).$$ 
		This yields $\OO(M_{t+1}^*)\subseteq\mndcl(E_{t+1}),$ and hence $\OO(M_{t+1}^*)=\mndcl(E_{t+1})$. Using induction on $n$, we conclude that $E_n$ generates $\OO(M_{n}^*)$ for all $n>t$, as desired. 
	\end{proof}
	
	%-------------------------------------------------------
	\section{Equivariant Hilbert bases}
	\label{sec:equi-Hilbert-bases}
	
	Given the classification of dual monoids of non-positive monoids in \Cref{duality: local non-positive}, the study of the dual sequence $\M^*=(M_n^*)_{n\ge 1}$ reduces to the case where the original chain $\M$ consists of positive monoids. In this situation, \Cref{duality:chain of positive monoids} shows that $M_n^*$ is eventually positive. A natural and central problem is therefore to describe $\Sym(n)$-equivariant Hilbert bases for $M_n^*$ when $n$ is sufficiently large. The goal of this section is to address this problem for stabilizing chains.
	
	A first approach, suggested by \Cref{M-W:local general}, is to start with a generating set $E_t$ of $\OO(M_t^*)$ and propagate it to higher dimensions via the construction
	\[
	E_n=\bigcup_{\ub\in E_t}\OO_{s,n}(\ub).
	\]
	Since $M_n^*$ is positive for all $n>r$, the same holds for $\OO(M_n^*)$. One might therefore expect that choosing $E_t$ to be the Hilbert basis of $\OO(M_t^*)$ would yield an equivariant Hilbert basis for $M_n^*$ for all $n\ge t$. However, the following example shows that this expectation fails in general.
	
	\begin{example}
		\label{ex:too-big}
		Let $\M$ be the chain considered in \Cref{example:M-W}(i). Then
		\[
		E_4=\{(1^{(4)}),\,(1^{(3)},2),\,(2^{(2)},3^{(2)}),\,(2,3^{(3)})\}
		\]
		is a minimal generating set, and hence the Hilbert basis, of $\OO(M_4^*)$. However, for any $n\ge4$, the set
		$$E_n=\bigcup_{\ub\in E_{4}} \OO_{2,n}(\ub)
		=\{(1^{(n)}),\,(1^{(n-1)},2)\}\cup\{(2^{(k)},3^{(n-k)})\mid k\in[n-2]\}$$
		is not  a $\Sym(n)$-equivariant Hilbert basis for $M_n^*$. Indeed, the element $(2^{(n-2)},3^{(2)})\in E_n$ admits the following decomposition:
		$$(2^{(n-2)},3^{(2)})=(1^{(n-1)},2)+(1^{(n-2)},2,1),$$ 
		showing that it is reducible.
	\end{example}
	
	This example indicates that the Hilbert basis of $\OO(M_t^*)$ may be too big for our purposes. Another natural choice for $E_t$ would be a $\Sym(t)$-equivariant Hilbert basis for $M_t^*$. However, as observed in \Cref{example:M-W}(i), such a set may be too small.
	
	The main contribution of this section is the discovery that, once $t$ is sufficiently large, choosing $E_t$ to be a $\Sym(t)$-equivariant Hilbert basis for $M_t^*$ does yield a $\Sym(n)$-equivariant Hilbert basis for $M_n^*$ for all $n>t$.
	
	\begin{theorem}\label{M-W:local 2r}
		Let $\M=(M_{n})_{n\geq 1}$ be a non-trivial stabilizing $\Sym$-invariant chain of positive monoids with $\ind(\M)=r$. Let $H_{4r+1}\subseteq\OO(M_{4r+1}^*)$ be a  $\Sym(4r+1)$-equivariant Hilbert basis for $M_{4r+1}^*$. Then for all $n> 4r+1$, the set
			$$H_n\defas\bigcup_{\ub\in H_{4r+1}} \OO_{2r,n}(\ub)$$
		is a $\Sym(n)$-equivariant Hilbert basis for $M_{n}^*$.
	\end{theorem}
	
	For nonnegative monoids, this result can be sharpened as follows.
	
	\begin{theorem}
		\label{M-W:nonnegative}
		Let $\M=(M_{n})_{n\geq 1}$ be a non-trivial stabilizing $\Sym$-invariant chain of monoids with $\ind(\M)=r$, and assume that $M_n \subseteq \ZZ^n_{\ge 0}$ for all $n\ge 1$. Let $H_{2r+1}\subseteq\OO(M_{2r+1}^*)$ be a  $\Sym(2r+1)$-equivariant Hilbert basis for $M_{2r+1}^*$. Then for all $n> 2r+1$, the set
		$$H_n\defas\bigcup_{\ub\in H_{2r+1}} \OO_{r,n}(\ub)$$
		is a $\Sym(n)$-equivariant Hilbert basis for $M_{n}^*$.
	\end{theorem}
	
	To prove these results, we first present a simple method for constructing equivariant Hilbert bases from generating sets. Recall that the set of irreducible elements of a monoid $M_n$ is denoted by $\irr(M_n)$. For a positive $\Sym(n)$-invariant monoid, this set coincides with its Hilbert basis (see \Cref{def:equivariant-Hilbert-basis}).
	
	\begin{lemma}
		\label{remark:equi Hb}
		Let $M_n\subseteq\ZZ^n$ be a $\Sym(n)$-invariant positive monoid, and suppose that $E_n$ is an arbitrary generating set of $\OO(M_{n})$. Then
		\[
		H_n\defas E_n\cap \irr(M_n)
		\]
		is the unique $\Sym(n)$-equivariant Hilbert basis for $M_n$ contained in $\OO(M_{n})$. In particular, $H_n$ is independent of the choice of the generating set $E_n$ of $\OO(M_{n})$.
	\end{lemma}
	
	\begin{proof}
		We first show that $\irr(M_n)=\Sym(n)(H_n)$. Since $\irr(M_n)$ is $\Sym(n)$-invariant by \Cref{prop:irreducible}, the inclusion $\irr(M_n)\supseteq\Sym(n)(H_n)$ is clear. For the reverse inclusion, note that $E_n$ is a $\Sym(n)$-equivariant generating set of $M_n$ by \Cref{prop:ordered-monoid}. It follows that $\irr(M_n)\subseteq \Sym(n)(E_n)$. Thus, for any $\ub\in\irr(M_n)$, there exists $\sigma\in\Sym(n)$ such that $\sigma(\ub)\in E_n$. This implies $\sigma(\ub)\in E_n\cap \irr(M_n)=H_n$, and hence $\ub\in\Sym(n)(H_n)$.
		
		Now assume that $H_n'\subseteq\OO(M_{n})$ is another set with $\irr(M_n)=\Sym(n)(H_n')$. Then for any $\ub\in H_n$, there exist $\vb\in H_n'$ and $\sigma\in\Sym(n)$ such that $\ub=\sigma(\vb)$. Since both $\ub$ and $\vb$ are ordered, we must have $\ub=\vb$. Therefore $H_n\subseteq H_n'$, and by symmetry $H_n'=H_n$. This proves the uniqueness of $H_n$. 
		
		Finally, the uniqueness of $H_n$ implies that it must be a minimal set with the property that $\irr(M_n)=\Sym(n)(H_n)$. Therefore, $H_n$ is a $\Sym(n)$-equivariant Hilbert basis for $M_n$.
	\end{proof}

	The following example illustrates the construction.
	
	\begin{example}
		We first consider the chain $\M$ from \Cref{example:M-W}(i). The set $E_{13}$ constructed there is a generating set of $\OO(M_{13}^*)$. As shown in \Cref{ex:too-big}, the element $(2^{(11)},3^{(2)})$ of $E_{13}$ is reducible. By using coordinate sums, one can check that all the remaining elements of $E_{13}$ are irreducible. Applying \Cref{remark:equi Hb}, we obtain the following $\Sym(13)$-equivariant Hilbert basis for $M_{13}^*$:
		$$H_{13}=E_{13}\cap \irr(M_{13}^*) =\{(1^{(13)}),\,(1^{(12)},2)\}\cup\{(2^{(k)},3^{(13-k)})\mid k\in[10]\}.$$
		By \Cref{M-W:local 2r}, the set
		\[
		H_n=\bigcup_{\ub\in H_{13}} \OO_{6,n}(\ub)
		=\{(1^{(n)}),\,(1^{(n-1)},2)\}\cup\{(2^{(k)},3^{(n-k)})\mid k\in[n-3]\}
		\]
		is a $\Sym(n)$-equivariant Hilbert basis for $M_{n}^*$ for all $n> 13$.
		
		Next, let $\M$ be the chain of nonnegative monoids considered in \Cref{example:M-W}(ii). As above, we find the following $\Sym(7)$-equivariant Hilbert basis for $M_{7}^*$:
		$$H_{7}=\{\eb_7,\,(-1^{(2)},5^{(5)}),\,(-1,0,3^{(5)}),\,(-1,1^{(6)})\}.$$
		By \Cref{M-W:nonnegative}, the set
		\[
		H_n=\bigcup_{\ub\in H_{7}} \OO_{3,n}(\ub)
		=\{\eb_n,\,(-1^{(2)},5^{(n-2)}),\,(-1,0,3^{(n-2)}),\,(-1,1^{(n-1)})\}
		\]
		is a $\Sym(n)$-equivariant Hilbert basis for $M_{n}^*$ for all $n> 7$.
	\end{example}
	
	Let us continue with the proofs of \Cref{M-W:local 2r,M-W:nonnegative}. To treat the case of nonnegative monoids, we need the following refinement of \Cref{cor:M-W-nonnegative}. An analogous construction for cones can be found in \cite[Theorem 4.4]{LR23}.
	
	\begin{lemma}
		\label{lem:generator-nonnegative}
		Let $\M=(M_{n})_{n\geq 1}$ be a non-trivial stabilizing $\Sym$-invariant chain of monoids with $\ind(\M)=r$, and assume that $M_n \subseteq \ZZ^n_{\ge 0}$ for all $n\ge 1$. Let $E_r$ be a generating set of $\OO(M_{r}^*)$. Define
		\begin{align*}
			E_{r+1}&=\{(u_1,\dots,u_r,u_r) \mid (u_1,\dots,u_r) \in E_r\} \cup \{\eb_{r+1}\},\\
			E_n&=\bigcup_{\ub\in E_{r+1}} \OO_{r,n}(\ub)
			\quad\text{for } n>r+1.
		\end{align*}
		Then the following statements hold for all $n\ge r+1$:
		\begin{enumerate}
			\item 
			$E_n$ is a generating set of $\OO(M_{n}^*)$.
			\item 
			$\eb_n\in \OO(M_{n}^*)\cap \irr(M_{n}^*)$.
			\item 
			If $\ub=(u_1,\dots,u_n)\in \OO(M_{n}^*)\cap \irr(M_{n}^*)\setminus\{\eb_n\}$, then $u_r=u_{r+1}=\cdots=u_n$.
		\end{enumerate} 
	\end{lemma}
	
	\begin{proof}
		(i) By \Cref{cor:M-W-nonnegative}, it suffices to show that $E_{r+1}$  generates $\OO(M_{r+1}^*)$. Consider the associated chain of cones $\C=(C_{n})_{n\geq 1}$, where $C_n=\cone(M_n)$. By \cite[Lemma 4.5]{LR23}, we have $E_{r+1}\subseteq C_{r+1}^*$ . Hence, by \Cref{prop:local-dualproperties},
		\[
		E_{r+1}\subseteq C_{r+1}^*\cap\ZZ^{r+1}=M_{r+1}^*,
		\]
		so $E_{r+1}\subseteq\OO(M_{r+1}^*)$, and therefore $\mndcl(E_{r+1})\subseteq\OO(M_{r+1}^*).$ For the reverse inclusion, let $\vb=(v_1,\dots,v_r,v_{r+1})\in \OO(M_{r+1}^*)$. Then $\hat{\vb}=(v_1,\dots,v_r)\in \OO(M_{r+1}^*)=\mndcl(E_{r})$. From the construction of $E_{r+1}$ it is easily seen that
		\[
		\tilde{\vb}=(v_1,\dots,v_r, v_r)\in\mndcl(E_{r+1}).
		\]
		Since $v_{r+1}\ge v_r$, we obtain
		\[
		\vb=\tilde{\vb}+(v_{r+1}-v_r)\eb_{r+1}\in\mndcl(E_{r+1}),
		\]
		and hence $\OO(M_{r+1}^*)\subseteq\mndcl(E_{r+1})$. Thus, equality holds.
		
		(ii) Since $M_n\subseteq \ZZ^n_{\ge 0}$, we have $\eb_n\in M_n^*$, and thus $\eb_n\in \OO(M_n^*)$. It remains to show that  $\eb_n\in \irr(M_{n}^*)$. Suppose $\eb_n=\ub+\vb$ with $\ub,\vb\in M_n^*\setminus\{\nub\}$. Then
		\begin{equation*}
			%\label{eq:e_n}
			1=s(\eb_n)=s(\ub)+s(\vb).
		\end{equation*}
		As $M_n^*$ is positive (by \Cref{duality:chain of positive monoids}), \Cref{thm:local positive and non-positive} implies that $s(\ub)$ and $s(\vb)$ are integers of the same sign, a contradiction. Hence, $\eb_n$ is irreducible.
		
		(iii) Since both $E_n$ and $\OO(M_n^*)$ generate $\OO(M_n^*)$, \Cref{remark:equi Hb} yields
		\[
		\OO(M_{n}^*)\cap \irr(M_{n}^*)=E_n\cap \irr(M_{n}^*)\subseteq E_n.
		\] 
		The conclusion follows from the construction of $E_n$.
	\end{proof}
	
	The proofs of \Cref{M-W:local 2r,M-W:nonnegative} rely crucially on the fact that the insertion operator $\OO_{i,m}$ preserves both irreducibility and reducibility under suitable conditions. We begin with the preservation of irreducibility.
	
	\begin{lemma}
		\label{lem:irreducibility-presevation}
		Let $\M=(M_{n})_{n\geq 1}$ be a non-trivial stabilizing $\Sym$-invariant chain of positive monoids with $\ind(\M)=r$. Let $s,t\in\NN$ satisfy $s\ge r$ and $t\ge s+r$. Assume that $\ub\in\OO(M_{t}^*)$. 
		If $\ub\in \irr(M_{t}^*)$, then $\OO_{s,n}(\ub)\subseteq \irr(M_{n}^*)$ for all $n>t$.
	\end{lemma}
	
	\begin{proof}
		By \Cref{duality:chain of positive monoids}, $M_{n}^*$ is a positive monoid for all $n>r$. Moreover, $\OO_{s,n}(\ub)\subseteq M_{n}^*$ by \Cref{lem:OO(u)}.  Using induction on $n$, it suffices to prove the statement for $n=t+1$. Assume to the contrary that there exists $\xb\in \OO_{s,t+1}(\ub)$ with $\xb\not\in\irr(M_{t+1}^*)$. Then 
		\begin{equation}
			\label{eq:reducible}
			\xb=\yb+\zb
		\end{equation}
		for some $\yb,\,\zb\in M_{t+1}^*\setminus\{\nub\}$. By definition, $\ub=\hat{\xb}$, where $\hat{\xb}$ is obtained from $\xb$ by deleting the its $(s+1)$-st coordinate. Applying this deletion to \eqref{eq:reducible}, we obtain
		\begin{equation}
			\label{eq:reducible-hat}
		\ub=\hat{\xb}=\hat\yb+\hat\zb,
	\end{equation}
		where $\hat{\yb},\,\hat{\zb}\in M_{t}^*$ by \Cref{lem:sucessive dual}. Since $\ub\in \irr(M_{t}^*)$, we may assume that $\hat\yb=\nub$. This implies $\yb=y_{s+1}\eb_{s+1}$ for some $y_{s+1}\in\ZZ\setminus\{0\}.$
		We distinguish the following cases:
		
		\emph{Case 1}: $M_r \nsubseteq \ZZ^r_{\ge 0}$ and $M_r \nsubseteq \ZZ^r_{\le 0}$.  
		Then $M_{t+1}$ is neither contained in $\ZZ^{t+1}_{\ge 0}$ nor in $\ZZ^{t+1}_{\le 0}$, since $M_r\subseteq M_{t+1}$. Consequently,
		\[
		a\eb_i \notin M_{t+1}^*
		\quad\text{for all } i\in [t+1] \text{ and } a\in\ZZ\setminus\{0\},
		\]
		which contradicts the fact that $\yb=y_{s+1}\eb_{s+1}\in M_{t+1}^*$.
		
		\emph{Case 2}: $M_r\subseteq \ZZ^r_{\ge 0}$. Then $M_k \subseteq \ZZ^k_{\ge 0}$ for all $k\ge 1$. Write $\ub=(u_1,\dots,u_t)$. Since $\ub\ne \eb_t$ (otherwise, $\xb=\eb_{t+1}\in\irr(M_{t+1}^*)$), \Cref{lem:generator-nonnegative} implies that $u_r=\cdots=u_t$. As $r\le s<t$, deleting the $s$-th coordinate of $\xb$ yields $\tilde{\xb}=\ub$. Thus, \eqref{eq:reducible} gives
		\[
		\ub=\tilde{\xb}=\tilde{\yb}+\tilde{\zb},
		\]
		where $\tilde{\yb},\,\tilde{\zb}\in M_{t}^*$ by \Cref{lem:sucessive dual}. Since $\ub\in \irr(M_{t}^*)$, either $\tilde{\yb}=\nub$ or $\tilde{\zb}=\nub$. If $\tilde{\yb}=\nub$, then together with $\hat\yb=\nub$, we must have $\yb=\nub$, a contradiction. If $\tilde{\zb}=\nub$, then $\zb=z_s\eb_s$ for some $z_{s}\in\ZZ\setminus\{0\}$, and hence $\hat{\zb}=z_s\eb_s$. From \eqref{eq:reducible-hat}, we obtain
		\[
		\ub=\hat\yb+\hat\zb=\hat\zb=z_s\eb_s,
		\]
		contradicting the assumption $\ub\in \OO(M_t^*)$.
		
		\emph{Case 3}: $M_r\subseteq \ZZ^r_{\le 0}$. This reduces to Case 2 by considering the chain $(-M_{n})_{n\geq 1}$.
	\end{proof}
	
	Next, we turn to the preservation of reducibility. Recall that the set of reducible elements of a monoid $M$ is denoted by $\re(M)$.
	
	\begin{lemma}
		\label{lem:reducibility-presevation}
		Let $\M=(M_{n})_{n\geq 1}$ be a non-trivial stabilizing $\Sym$-invariant chain of positive monoids with $\ind(\M)=r$. Let $s,t\in\NN$, and assume that $\ub=(u_1,\dots,u_t)\in\OO(M_{t}^*)$.
		If $\ub\in \re(M_{t}^*)$, then $\OO_{s,n}(\ub)\subseteq \re(M_{n}^*)$ for all $n>t$, provided that one of the following holds:
		\begin{enumerate}
			\item $s\ge 2r$ and $t\ge s+2r+1$;
			\item $s\ge r$, $t\ge s+r+1$, $u_r=u_t$, and $M_r \subseteq \ZZ^r_{\ge 0}$.
		\end{enumerate}
	\end{lemma}
	
	\begin{proof}
		By induction on $n$, it suffices to consider the case $n=t+1$. Since $\ub\in \re(M_{t}^*)$, we can write
		\begin{equation}
			\label{eq:u=v+w}
			\ub=\vb+\wb
		\end{equation}
		with $\vb=(v_1,\dots,v_{t}),\, \wb=(w_1,\dots,w_{t})\in M_{t}^*\setminus \{\nub\}$. Let $\ub'\in\OO_{s,t+1}(\ub)$, so that
		\[
		\ub'=(u_1,\dots,u_s,a,u_{s+1},\dots,u_t)
		\quad\text{for some } a\in[u_s,\,u_{s+1}].
		\] 
		We aim to construct $\vb',\wb'\in M_{t+1}^*\setminus\{\nub\}$ such that $\ub'=\vb'+\wb'$.
		Let $\vb^+=(v_1^+,\dots,v_t^+)$ and $\wb^+=(w_1^+,\dots,w_t^+)$ denote the non-decreasing rearrangements of $\vb$ and $\wb$, respectively.
		
		(i)  Suppose $s\ge 2r$ and $t\ge s+2r+1$. We claim that 
		$$v^+_r + w^+_r \le u_{s}\le u_{s+1}\le v^+_{t-r}+w^+_{t-r}.$$
		Indeed, since $\ub\in\OO(M_{t}^*)$, it follows from \eqref{eq:u=v+w} that 
		$$v_i+w_i\le u_{2r}\le u_{s}\le u_{s+1}\le u_{t-2r}\le v_j+w_j$$
		for all $i\in [2r]$ and $j\in [t-2r,\,t]$.
		%Thus, $\vb+\wb$ has at least $2r$ coordinates that does not exceed $u_{s}$ and also at least $2r+2$ coordinates that is at least $u_{s+1}$.
		Now suppose, to the contrary, that $v^+_r + w^+_r > u_{s}$. Then for every $i\in[2r]$,
		$$v_i+w_i\le u_{s}<v^+_r + w^+_r,
		$$ 
		so at least one of the inequalities $v_i < v_r^+$ or $w_i < w_r^+$ must hold. Hence,
		\[
		[2r]\subseteq\{i\mid v_i<v^+_r\}\cup\{i\mid w_i<w^+_r\}.
		\] 
		However, by definition of the ordered coordinates, each of the sets on the right has cardinality at most $r-1$. Thus their union has size at most $2r-2$, which is impossible. Therefore, $v_r^+ + w_r^+ \le u_s$. The inequality $u_{s+1} \le v_{t-r}^+ + w_{t-r}^+$ follows by a similar argument applied to the indices $j\in[t-2r,\,t]$.
		
		Since $a\in[u_s,\,u_{s+1}]$, \Cref{lem:ineq} implies that there exist $b,\, c \in \NN$ with $b\in [v^+_r,\,v^+_{t-r}]$ and $c\in [w^+_r,\,w^+_{t-r}]$ such that 
		$a=b+c.$
		Define
		\[
		\vb'=(v_1,\dots,v_s,b,v_{s+1},\dots,v_t)
		\quad\text{and}\quad
		\wb'=(w_1,\dots,w_s,c,w_{s+1},\dots,w_t).
		\]
		Then 
		$\ub'=\vb'+\wb'.$
		It remains to show that $\vb',\, \wb'\in M_{t+1}^*\setminus \{\nub\}$.
		Since  $b\in [v^+_r,\,v^+_{t-r}]$, we have $(\vb')^+\in \OO_{k,t+1}(\vb^+)$ for some $k \in [r,\, t-r]$. As $k\ge r$ and $t\ge k+r$, it follows from \Cref{lem:OO(u)} that $\OO_{k,t+1}(\vb^+)\subseteq M_{t+1}^*$. Hence, $(\vb')^+ \in M_{t+1}^*,$ and so $\vb'\in M_{t+1}^*$. Similarly, $\wb'\in M_{t+1}^*$. Finally, we have $\vb',\wb'\ne \nub$ since $\vb,\wb\ne \nub$.
		
		(ii) Since $s, s+1, 2r \in [r,\, t]$ and $\ub \in \OO(M_t^*)$ with $u_r = u_t$, we obtain $u_{2r}=u_s=a=u_{s+1}$. Arguing as in (i), this gives $v^+_r + w^+_r \le u_{2r}=a.$ 
		We show that $a\le v^+_{t-1} + w^+_{t-1}$. From 
		\eqref{eq:u=v+w} and $u_r = \cdots = u_t$, we have
		\[
		a=v_r+w_r=v_{r+1}+w_{r+1}=\cdots=v_t+w_t.
		\]  
		Reordering if necessary, assume that
		\[
		v_r\le v_{r+1}\le \cdots\le v_t.
		\]
		Then
		\[
		w_t\le w_{t-1}\le \cdots\le w_r.
		\]
		It follows that $v^+_{t-1} \ge v_{t-1}$ and $w^+_{t-1}\ge w_{r+1}$. 
		Hence,
		\[
		v^+_{t-1} + w^+_{t-1}\ge v_{t-1}+w_{r+1}\ge v_{r+1}+w_{r+1}=a.
		\]
		Now by \Cref{lem:ineq}, we can write $a = b + c$ with $b\in [v^+_r,\,v^+_{t-1}]$ and $c\in [w^+_r,\,w^+_{t-1}]$. Defining $\vb'$ and $\wb'$ as in (i), we obtain $\ub' = \vb' + \wb'$ with $\vb', \wb' \ne \nub$.
		Finally, since $b \in [v_r^+, v_{t-1}^+]$, we have $(\vb')^+ \in \OO_{k,t+1}(\vb^+)$ for some $k \in [r,\, t-1]$. As $M_r \subseteq \ZZ_{\ge 0}^r$, \Cref{lem:OO(u)} implies $(\vb')^+ \in M_{t+1}^*$, and hence $\vb' \in M_{t+1}^*$. Similarly, $\wb' \in M_{t+1}^*$. This completes the proof.
	\end{proof}
	
	We are now ready to prove \Cref{M-W:local 2r,M-W:nonnegative}.
	
	\begin{proof}[Proof of \Cref{M-W:local 2r}]
		Let $E_{2r}$ be a generating set of  $\OO(M_{2r}^*)$ and define
		\[  
		E_n=\bigcup_{\ub\in E_{2r}} \OO_{r,n}(\ub)
		\quad\text{for } n>2r.
		\]
		Then by \Cref{cor:M-W-simplified}, $E_n$ generates $\OO(M_n^*)$ for all $n>2r$.
		Since $H_{4r+1}$ is the  $\Sym(4r+1)$-equivariant Hilbert basis for $M_{4r+1}^*$ contained in $\OO(M_{4r+1}^*)$, \Cref{remark:equi Hb} yields 
		$$H_{4r+1} = E_{4r+1}\cap \irr(M_{4r+1}^*).$$ 
		We need to show that
		\[
		H_{n} = E_{n}\cap \irr(M_{n}^*)
		\quad\text{for all }n>4r+1.
		\]
		Fix $n>4r+1$. By \Cref{lem:irreducibility-presevation}, irreducibility is preserved under $\OO_{2r,n}$, and hence 
		$$H_{n} =\bigcup_{\ub\in H_{4r+1}} \OO_{2r,n}(\ub)\subseteq \OO(M_{n}^*)\cap \irr(M_{n}^*).$$ 
		This gives
		\[
		H_{n} \subseteq E_{n}\cap \irr(M_{n}^*)
		\]
		since $\OO(M_n^*)\cap \irr(M_n^*)=E_n\cap \irr(M_n^*)$ by \Cref{remark:equi Hb}.
		For the reverse inclusion, it suffices to prove that $E_n\setminus H_n\subseteq\re(M_{n}^*).$ Let $\vb=(v_1,\dots,v_n)\in E_n\setminus H_n$, and set 
		$$\wb=(v_1,\dots,v_{2r},v_{n-2r},\dots,v_n).$$
		Then $\wb\in E_{4r+1}$ and $\vb\in\OO_{2r,n}(\wb)$. Moreover, since $H_n=\bigcup_{\ub\in H_{4r+1}} \OO_{2r,n}(\ub)$ and $\vb\not\in H_{n}$, we must have $\wb\not\in H_{4r+1}$. Therefore, $\wb\in E_{4r+1}\setminus H_{4r+1}$, and hence $\wb\in \re(M_{4r+1}^*)$ because $H_{4r+1} = E_{4r+1}\cap \irr(M_{4r+1}^*)$. Applying \Cref{lem:reducibility-presevation}(i), we get $\vb\in \re(M_n^*)$, as required.
	\end{proof}
	
	\begin{proof}[Proof of \Cref{M-W:nonnegative}]
		We use similar reasoning as in the proof of \Cref{M-W:local 2r}. Let $E_r$ be a generating set of $\OO(M_r^*)$, and define $E_n$ for $n\ge r+1$ as in \Cref{lem:generator-nonnegative}. Then \Cref{remark:equi Hb} gives
		\[
		H_{2r+1} = E_{2r+1}\cap \irr(M_{2r+1}^*),
		\]
		and we aim to prove
		\[
		H_{n} = E_{n}\cap \irr(M_{n}^*)
		\quad\text{for all } n>2r+1.
		\]
		The inclusion $H_{n} \subseteq E_{n}\cap \irr(M_{n}^*)$ follows as before. For the reverse inclusion, note that $\eb_{2r+1}\in H_{2r+1}$ by \Cref{lem:generator-nonnegative}, hence $\eb_n\in H_n$. Let $\vb=(v_1,\dots,v_n)\in E_n\setminus H_n$. By construction of $E_n$, we have
		\[
		v_r=v_{r+1}=\cdots=v_n.
		\]
		Set $\wb=(v_1,\dots,v_{2r+1})$. Arguing as in the previous proof, one gets $\wb\in \re(M_{2r+1}^*)$.
		Finally, by \Cref{lem:reducibility-presevation}(ii), we obtain $\vb\in\OO_{r,n}(\wb)\subseteq\re(M_{n}^*)$, which completes the proof.
	\end{proof}
	
	We conclude with a discussion of Hilbert bases for the ordered dual monoids $\OO(M_n^*)$. We first show that the construction in \Cref{M-W:local general} preserves the minimality of generating sets.
	
	\begin{proposition}
		\label{lem:minimal set}
		Let $\M=(M_{n})_{n\geq 1}$ be a stabilizing $\Sym$-invariant chain of monoids with $\ind(\M)=r$, and let $G_r$ be a $\Sym(r)$-equivariant generating set for $M_r$. Choose integers $s<t$ such that $s \ge \mu_+(G_r)$ and $t\ge \max \{s+\mu_-(G_r),r\}$. Suppose $E_{t}$ is a generating set for $\OO(M_{t}^*)$.
		For $n>t$, define
		$$E_n=\bigcup_{\ub\in E_{t}} \OO_{s,n}(\ub).$$ 
		Then the following are equivalent:
		\begin{enumerate}
			\item $E_t$ is a minimal generating set of $\OO(M_{t}^*)$;
			\item $E_m$ is a minimal generating set of $\OO(M_{m}^*)$ for some $m> t$;
			\item $E_n$ is a minimal generating set of $\OO(M_{n}^*)$ for all $n\ge t$.
		\end{enumerate}
	\end{proposition}

	\begin{proof}
		By \Cref{M-W:local general}, each $E_n$ generates $\OO(M_n^*)$ for $n>t$. It therefore suffices to prove that minimality is preserved.
		
		(i) $\Rightarrow$ (iii): Assume $E_t$ is minimal. By induction on $n$, it is enough to show that $E_{t+1}$ is also minimal. Suppose not. Then there exist $\wb\in E_{t+1}$ and $\vb_1,\dots,\vb_k\in E_{t+1}$ such that
		$$\wb=\sum_{i=1}^k a_i\vb_i
		\quad\text{with } a_i\in\ZZ_{\ge0}.
		$$ 
		Deleting the $(s+1)$-st coordinate yields $\hat\wb=\sum_{i=1}^k a_i\hat\vb_i$, where by construction, $\hat\wb,\,\hat\vb_i\in E_{t}$ for all $i\in [k]$. This contradicts the minimality of $E_t$. Hence, $E_{t+1}$ is minimal.
		
		The implication (iii) $\Rightarrow$ (ii) is obvious.
		
		(ii) $\Rightarrow$ (i): It suffices to show that minimality of $E_{t+1}$ implies minimality of $E_t$. Suppose $E_t$ is not minimal. Then there exist $\wb\in E_t$ and $\vb_1,\dots,\vb_k\in E_t$ such that
		$$\wb=\sum_{i=1}^k a_i\vb_i
		\quad\text{with } a_i\in\ZZ_{\ge0}.
		$$  
		For each $i\in [k]$, choose $\vb_i'\in \OO_{s,t+1}(\vb_i)$, and set
		$$\wb'=\sum_{i=1}^k a_i\vb_i'.$$ 
		Then $\wb'\in \OO_{s,t+1}(\wb)\subseteq E_{t+1}$, while each $\vb_i'\in E_{t+1}$. This contradicts the minimality of $E_{t+1}$. Hence, $E_t$ is minimal.
	\end{proof}
	
	Applying the previous result to chains of positive monoids, we immediately obtain the following characterization of Hilbert bases for $\OO(M_n^*)$.
	
	\begin{corollary}
		\label{cor:ordered-Hilbert-bases}
		Let $\M=(M_{n})_{n\geq 1}$ be a non-trivial stabilizing $\Sym$-invariant chain of positive monoids with $\ind(\M)=r$. Let $s,t\in\NN$ satisfy $s \ge r$ and $t\ge s+r$. Suppose $E_{t}$ is a generating set of $\OO(M_{t}^*)$.
		For $n>t$, define
		$$E_n=\bigcup_{\ub\in E_{t}} \OO_{s,n}(\ub).$$ 
		Then the following are equivalent:
		\begin{enumerate}
			\item $E_t$ is the Hilbert basis of $\OO(M_{t}^*)$;
			\item $E_m$ is the Hilbert basis of $\OO(M_{m}^*)$ for some $m> t$;
			\item $E_n$ is the Hilbert basis of $\OO(M_{n}^*)$ for all $n\ge t$.
		\end{enumerate}
	\end{corollary}
	
	\begin{example}
		Consider again the chain $\M$ from \Cref{example:M-W}(i). The set $E_4$ given there is the Hilbert basis of $\OO(M_{4}^*)$. For $n>4$, while the set $E_n=\bigcup_{\ub\in E_{4}} \OO_{2,n}(\ub)$ is not an equivariant Hilbert basis for $M_{n}^*$ (see \Cref{ex:too-big}), it is the Hilbert basis of $\OO(M_{n}^*)$ by \Cref{cor:ordered-Hilbert-bases}. 
	\end{example}

	%-------------------------------------------------------
	
\end{document}